\documentclass[11pt]{article}
\usepackage{amsmath}
\usepackage{amsfonts}
\usepackage{amsthm}
\usepackage{amssymb, setspace}
\usepackage{color,graphicx,epsfig,geometry,hyperref,fancyhdr}
\usepackage[T1]{fontenc}
\usepackage{float}
\usepackage{subfig}
\usepackage{commath}
\usepackage{bbm}
\usepackage{mathrsfs,fleqn}
\newtheorem{theorem}{Theorem}[section]

\newtheorem{lemma}{Lemma}[section]

\newcommand{\N}{\mathbb{N}}

\newcommand{\R}{\mathbb{R}}
\newcommand{\C}{\mathbb{C}}

\newcommand{\brac}[1]{\left\{#1\right\}}

\graphicspath{{paper-pics/}}

\begin{document}

\begin{flushleft}
\Large 
\noindent{\bf \Large Reconstruction of small and extended regions in EIT with a Robin transmission condition}
\end{flushleft}

\vspace{0.2in}

{\bf  \large Govanni Granados and Isaac Harris}\\
\indent {\small Department of Mathematics, Purdue University, West Lafayette, IN 47907 }\\
\indent {\small Email:  \texttt{ggranad@purdue.edu}  and \texttt{harri814@purdue.edu} }\\


\begin{abstract}
\noindent We consider an inverse shape problem coming from electrical impedance tomography with a Robin transmission condition. In general, a boundary condition of Robin type models corrosion. In this paper, we study two methods for recovering an interior corroded region from electrostatic data. We consider the case where we have small volume and extended regions. For the case where the region has small volume, we will derive an asymptotic expansion of the current gap operator and prove that a MUSIC-type algorithm can be used to recover the region. In the case where one has an extended region, we will show that the regularized factorization method can be used to recover said region. Numerical examples will be presented for both cases in two dimensions in the unit circle. 
\end{abstract}



\section{Introduction}
The problem we consider in this paper is motivated by electrical impedance tomography (EIT). The goal in EIT is to reconstruct interior defects from the measured electrostatic data on the surface of an object. This corresponds to an inverse shape problem where the knowledge of the solution to a boundary value problem is used to recover unknown interior regions. See \cite{eit-review,eit-review-amend,EIT-cheney,MUSIC-Hanke2,mueller-book} for more discussion on the theory and applications of EIT. This is a very useful imaging method of nondestructive testing. In the case of medical imaging EIT is a fast, non-invasive, and cost effective way to detect abnormalities in a patient. Here we are interested in reconstructing a subregion where a transmission condition is imposed. This transmission condition is given by a Robin type boundary condition which models corrosion in the case of EIT.

In this paper, we will assume that voltage is applied to the known exterior boundary and the induced current is measured also on the exterior boundary. Mathematically speaking, we are interested in deriving an algorithm for recovering the unknown region given the Dirichlet-to-Neumann mapping on the exterior boundary. In \cite{eit-transmission1,eit-transmission2} the authors have studied the inverse parameter problem for the EIT problem with with a Robin transmission condition. In the aforementioned papers, the authors studied the uniqueness, stability and numerical reconstruction for the inverse parameter problem using the Neumann-to-Dirichlet mapping, whereas we study the inverse shape problem, proving that the Dirichlet-to-Neumann mapping uniquely recovers the region of interest. We also derive imaging functionals for reconstructing the region.

In order to solve the inverse shape problem, we will develop two qualitative reconstruction methods. A disadvantage of using iterative methods is that they require a ``good'' initial estimate for the unknown region and/or parameters to insure that the iterative process will converge to the unique solution of the inverse problem. To avoid requiring any additional a priori knowledge of the region of interest we will analyze two qualitative methods. These methods usually require little to no a priori knowledge of the region of interest denoted $D \subset \R^d$. This is done by connecting the region of interest to the range of the {measured} Dirichlet-to-Neumann mapping. Therefore, we can characterize the unknown region $D$ by the spectral/singular-value decomposition of the measured data operator. This makes the numerical implementation of these methods computationally simple since one only needs to compute the spectral/singular-value decomposition of the discretized operator, which is more cost effective in contrast to the steps required to derive an effective iterative method i.e. solving (multiple) adjoint problems at each step in the iteration.

Here we will consider a {\bf MU}ltiple {\bf SI}gnal {\bf C}lassification (MUSIC)-type algorithm for recovering small volume regions. This method has been used in many imaging modalities such as acoustic \cite{MUSIC-ammari-scattering,MUSIC-jake,MUSIC-park}, electromagnetic \cite{MUSIC-EM1,MUSIC-EM2,MUSIC-parkEM}, and elastic  \cite{MUSIC-elastic,MUSIC-elastic2} inverse scattering. To derive the MUSIC algorithm, we will need to exploit the fact that the regions of interest have small volume. To this end, we will need to derive a suitable asymptotic expansion for the  Dirichlet-to-Neumann mapping associated with this problem. We will also consider the regularized factorization method for solving the inverse problem with extended regions of interest. This regularized variant of the factorization method was initially studied in \cite{harris1} for a similar problem coming from diffuse optical tomography. This method is based on the analysis in \cite{arens,GLSM,EIT-FM,kirschbook}. The analysis we present here for the small volume and extended regions works in both $\R^2$ or $\R^3$ making these methods robust in their applications.

The rest of the paper is organized as follows. In Section \ref{dp-ip} we will rigorously define the direct and inverse problem under consideration. Here we will first consider the wellposedness of the direct problem and define the current gap operator $(\Lambda - \Lambda_0)$ that will be used to derive our imaging functionals. Then, we consider the asymptotic expansion of the current gap operator in Section \ref{Asymptotic Section}. Using the asymptotic expansion we will derive the MUSIC algorithm for recovering the components of the region $D$. We will then consider the case for an extended region in Section \ref{Factorization section}. To this end, we further analyze $(\Lambda - \Lambda_0)$ in order to derive a suitable factorization to apply the theory in \cite{harris1} to derive an efficient imaging functional to reconstruct the shape of $D$. In Sections \ref{Asymptotic Section} and  \ref{Factorization section} numerical examples are presented in $\R^2$ to validate the analysis of the studied imaging functionals. Lastly, in Section \ref{end} we will end the paper by summarizing the results as well as giving an outlook on possible future projects in this direction. 

\section{The direct and inverse problem}\label{dp-ip}
We begin by considering the direct problem associated with the electrostatic imaging of a defective region with a Robin transmission condition on its boundary. Assume that $ \Omega \subset \mathbb{R}^d$ is a simply connected open set with Lipshitz boundary $\partial \Omega$. Let $D \subset \Omega$ be a (possibly multiple) connected open set with class $\mathcal{C}^2$ boundary $\partial D$. We assume that $\text{dist}(\partial \Omega , \overline{D}) > 0$. For the material with defective region(s), we define $u \in H^1 (\Omega)$ as the solution to
\begin{equation}\label{mpde}
- \Delta u = 0 \quad \text{in} \quad \Omega \textbackslash \partial D \quad \text{with} \quad u \big|_{\partial \Omega} = f \quad \text{and} \quad [\![\partial_\nu u ]\!] \big|_{\partial D} = \gamma u \big|_{\partial D}
\end{equation}
where
 $$ [\![\partial_\nu u ]\!] \big|_{\partial D} := ( \partial_{\nu} u^{+} - \partial_{\nu} u^{-}) \big|_{\partial D}$$
for a given $f \in H^{1/2} ( \partial \Omega )$. For the rest of the paper, we let $\nu$ denote the unit outward normal on the boundaries $\partial D$ and $\partial \Omega$. 

Here, the function $u$ is the electrostatic potential for the defective material. The `+' notation represents the trace taken from $\Omega \setminus \overline{D}$ and the `$-$' notation represents the trace taken from $D$. This Robin transmission condition in \eqref{mpde} models the corrosion of $\partial D$ and states that the jump in current across this boundary is proportional to the electrostatic potential $u$. Furthermore, since we assume that $u \in H^1 (\Omega)$, it is known that $ [\![u ]\!]  \big|_{\partial D} = 0$. This comes from the fact that any function in $H^1 (\Omega)$ has equal interior trace `$-$' and exterior trace `+' on any subdomain of $\Omega$. The analysis in the following sections holds for dimensions $d=2$ and $d=3$. 

We assume that the transmission parameter $\gamma \in L^{\infty} ( \partial D )$. For analytical purposes of well-posedness of the direct problem and the upcoming analysis of the inverse problem, we assume for the rest of the paper that there are constants $\gamma_{\text{max}}$ and $\gamma_{\text{min}}$ such that 
$$0< \gamma_{\text{min}} \leq  \gamma (x) \leq \gamma_{\text{max}}  \quad \text{for a.e.} \quad x \in \partial D. $$ 
We now begin by showing that the boundary value problem (1) is well posed for any given $f \in H^{1/2} ( \partial \Omega)$. To this end, we consider Green's 1st Theorem on the region $\Omega \backslash \overline{D}$ 
$$\int_{\Omega \backslash D} \nabla u \cdot \nabla \overline{\varphi} \, \text{d}x = \int_{\partial \Omega}  \overline{\varphi}  \partial_{\nu} u \, \text{d}s - \int_{D}  \overline{\varphi} \partial_{\nu} u^{+} \, \text{d}s $$
as well as Green's 1st Theorem on the region $D$
$$ \int_{D} \nabla u \cdot \nabla \overline{\varphi} \, \text{d}x = \int_{\partial D}  \overline{\varphi} \partial_{\nu} u^{-} \, \text{d}s $$ 
for any test function $\varphi \in H^1 (\Omega)$. The variational formulation for \eqref{mpde} is given by  adding these two equations
\begin{equation}\label{vf}
\int_{\Omega} \nabla u \cdot \nabla \overline{\varphi} \, \text{d}x =  \int_{\partial \Omega} \overline{\varphi} \partial_{\nu} u  \, \text{d}s -  \int_{\partial D} \overline{\varphi} \gamma u  \, \text{d}s
\end{equation}
where we have used the Robin transmission condition on $\partial D$. Before proceeding, we let $u_0 \in H^{1} ( \Omega)$ be the harmonic lifting of the Dirichlet data such that  
\begin{equation}\label{hl}  
- \Delta u_0 = 0 \quad \text{in} \quad \Omega \quad \text{with} \quad u_0 \big|_{\partial \Omega} = f.
\end{equation}
We make the ansatz that the solution can be written as $u = v + u_0$ with the function $v \in H_{0}^{1} (\Omega)$ where we define the space as 
$$H_{0}^{1} (\Omega) = \{ \varphi \in H^{1} (\Omega) \,\, : \,\, \varphi |_{\partial \Omega} = 0\}$$
with the same norm as $H^{1} (\Omega)$. Thus, the variational formulation of \eqref{mpde} with respect to $v$ is given by
\begin{equation} \label{ses}
A(v , \varphi) = - A(u_0 , \varphi) \quad  \text{for all} \quad \varphi \in H_{0}^{1} (\Omega)
\end{equation}
where the sesquilinear form $A(\cdot , \cdot): H_{0}^{1} (\Omega) \times H_{0}^{1} ( \Omega ) \mapsto \mathbb{C}$ is given by 
$$ A(v , \varphi) = \int_{\Omega} \nabla v \cdot \nabla \overline{\varphi} \, \text{d}x + \int_{\partial D} \gamma \, v \, \overline{\varphi} \, \text{d}s.$$
It is clear that the sesqulinear form is bounded whereas the coercivity on $H^1_0(\Omega)$ can be shown by the assumptions on $\gamma$ as well as the Poincar\'{e} inequality. We also have that $A(u_0 , \varphi)$ is a conjugate linear and bounded functional acting on  $H_{0}^{1} (\Omega)$ and using the Trace Theorem we have that 
$$| A(u_0 , \varphi ) | \leq C \|{f}\|_{H^{1/2} (\partial \Omega)} \|{\varphi}\|_{H^{1} (\Omega)}.$$ By the Lax-Milgram lemma, there is a unique solution $v$ to \eqref{ses} satisfying 
$$\|{v}\|_{H^{1} (\Omega)} \leq C \|{f}\|_{H^{1/2} (\partial \Omega)}.$$ 
Using the sesquilinear form $A(\cdot \,  , \cdot)$, we can show that the solution $u$ for equation \eqref{mpde} is unique just as in \cite{harris2}, which implies that equation \eqref{mpde} is well-posed. The above analysis gives the following result.

\begin{theorem}\label{soloperator}
The solution operator corresponding to the boundary value problem \eqref{mpde} $f \mapsto u$ is a bounded linear mapping from $H^{1/2} ( \partial \Omega)$ to $H^{1}(\Omega)$.
\end{theorem}

We now assume that the voltage $f$ is applied to the outer boundary $\partial \Omega$ and the measured data is given by the current $\partial_{\nu}u$. From the knowledge of the measured currents, we wish to derive two different types of qualitative sampling algorithms to determine the defective region $D$ without the knowledge of the transmission parameter $\gamma$ and with little to no prior knowledge on the number of regions. To this end, we define the data operator that will be studied in the following sections to derive our algorithms. Note that the function $u_0$ is the electrostatic potential for the healthy material and is known since the outer boundary is known. By the linearity of the partial differential equation and boundary conditions on $\partial \Omega$ and $\partial D$, we have that the voltage to electrostatic potential mappings 
$$ f \longmapsto u \quad \text{and} \quad f \longmapsto u_0$$
are bounded linear operators from $H^{1/2} (\partial \Omega)$ to $H^{1} (\Omega)$. We now define the \textit{Dirichlet-to-Neumann} (DtN) mappings as 
$$ \Lambda \enspace \text{and} \enspace \Lambda_0 : H^{1/2} ( \partial \Omega) \longrightarrow H^{-1/2} ( \partial \Omega)$$ 
where 
$$ \Lambda f = \partial_{\nu} u \big|_{\partial \Omega} \quad \text{and} \quad \Lambda_{0} f = \partial_{\nu} u_0 \big|_{\partial \Omega}.$$
By appealing to Theorem \ref{soloperator} and the well-posedness of \eqref{hl}, we have that the DtN mappings are bounded linear operators by Trace Theorems. Our main goal is to solve the \textit{inverse shape problem} of recovering the boundary $\partial D$ from the knowledge of the difference of the DtN mappings. This difference on the outer boundary $\partial \Omega$ is the current gap imposed on the system by the presence of the defective region $D$. By analyzing the data operator $\Lambda - \Lambda_0$, we wish to solve the inverse shape problem by deriving computationally simple algorithms to detect the defective region(s) via qualitative methods.

\section{\textbf{Recovering Regions of Small Volume}}\label{Asymptotic Section}
In this section, we will develop the MUSIC Algorithm for solving the inverse problem under consideration. The goal is to first, derive an asymptotic expansion of the current gap operator $\Lambda - \Lambda_0$. Then, being motivated by analysis in \cite{MUSIC-sweep,MUSIC-armin}, we will derive an analog of the multi-static response matrix derived from the current gap operator for this inverse shape problem. The asymptotic analysis here is different from the typical techniques used in \cite{MUSIC-ammari-eit,MUSIC-Hanke,MUSIC-armin}. See for e.g. \cite{shari2,shari1} for application to inversion from the asymptotic analysis. In the aforementioned papers, the authors use asymptotic results for the inverse associated with the double--layer potential operator. Here our analysis is based on a representation of the current gap operator using boundary integrals.

\subsection{\textbf{MUSIC Algorithm}}
We now begin our analysis of the asymptotic expansion of the current gap operator $\Lambda - \Lambda_0$ applied to the known voltage $f \in H^{1/2} (\partial \Omega)$. The operator $\Lambda$ is known from measurements and $\Lambda_0$ is given from direct calculations. The asymptotic analysis will allow us to reconstruct the unknown region in the case when $|D| = \mathcal{O}(\epsilon^d)$, where $d$ = 2 or 3 is the dimension, i.e. when the region has small volume. We let 
\begin{equation}\label{sball}
D = \bigcup_{j=1}^{J} D_j \quad \text{with} \quad D_j = (x_j + \epsilon B_j) \quad \text{such that} \quad \text{dist}(x_i , x_j) \geq c_0>0
\end{equation}
for $ i \neq j$ where the parameter $0< \epsilon \ll 1$ and $B_j$ is a domain with $\mathcal{C}^2$ boundary centered at the origin such that $|B_j| = \mathcal{O}(1)$. We also assume that the regions $D_j$ are disjoint. Now, we define the Dirichlet Green's function for the negative Laplacian for the known domain $\Omega$ as $\mathbb{G} (\cdot \, , z) \in H_{loc}^{1} ( \Omega \backslash \brac{z})$, which is the unique solution to the boundary value problem 
$$ - \Delta \mathbb{G}(\cdot \,  , z) = \delta (\cdot - z) \enspace \text{in} \enspace \Omega \quad \text{and} \quad \mathbb{G}(\cdot  \, , z ) \big|_{\partial \Omega} =0.$$ 
For any fixed $z \in \Omega$, we use Green's 2nd Theorem similarly as in Section \ref{dp-ip} to obtain the representation 
\begin{align*}
-(u - u_0)(z) = \int_{\Omega} (u - u_0)(x) \Delta \mathbb{G}(x,z) \, \text{d}x &= \int_{\partial D} \mathbb{G}(x , z ) [\![\partial_\nu u(x) ]\!] \, \text{d}s(x)  \\
	&= \int_{\partial D} \mathbb{G}(x,z) \gamma(x) u(x) \, \text{d}s(x)
\end{align*}
where we used the Robin transmission  condition on the interior boundary $\partial D$. By taking the normal derivative, we have that for all $z \in \partial \Omega$
\begin{align}\label{aexp}
(\Lambda - \Lambda_0 ) f(z) & = - \int_{\partial D} \gamma(x) u(x) \partial_{\nu (z)} \mathbb{G} (x , z) \, \text{d}s(x) \nonumber \\
					    &\hspace{-0.4in}= - \int_{\partial D} \gamma(x) u_0(x) \partial_{\nu (z)} \mathbb{G} (x , z) \, \text{d}s(x) - \int_{\partial D} \gamma(x) (u - u_0 )(x) \partial_{\nu (z)} \mathbb{G} (x , z) \, \text{d}s(x)  
\end{align}
where the integrands are continuous with respect to $z \in \partial \Omega$ and $\partial_{\nu (z)}$ denotes the normal derivative with respect to $z$. We claim that \eqref{aexp} is dominated by the first integral. In other words, the current gap for $f$ at any $z \in \partial \Omega$ can be  approximated by using the harmonic lifting $u_0$ restricted to the inner boundary instead of the unknown electrostatic potential $u$. 

The following estimates will help us in our asymptotic analysis of \eqref{aexp}. We will use the following Trace Theorem (see for e.g. Theorem 1.6.6 in \cite{trace-ref})
\begin{equation}\label{trace}
\|{\varphi}\|_{L^{2}(\partial D)}^{2} \leq C \|{\varphi}\|_{L^{2}(D)} \|{\varphi}\|_{H^{1} (D)}
\end{equation}
for all $\varphi \in H^{1}(\Omega)$ and all $D \subset \Omega$. A simple change of variables shows that the constant in \eqref{trace} is independent of the parameter $\epsilon$. We also use the estimate derived in Theorem 3.1 of \cite{cakoni2}, which states that for all $\varphi \in H^{1}(\Omega)$ with $D \subset \Omega$ such that $|D| = \mathcal{O}  (\epsilon^d )$, we have that 
\begin{equation}\label{sobo}
\|{\varphi}\|_{L^{2}(D)} \leq C \, \epsilon^{\frac{d}{2}\left( 1- \frac{2}{p} \right)} \|{\varphi}\|_{H^{1}(\Omega)}
\end{equation}
where $p \geq 2$ in $d=2$ and $2 \leq p \leq 6$ in $d=3$. This estimate is proven by using the Sobolev embedding of $H^{1}(\Omega)$ into $L^p(\Omega)$ (see for e.g. Chapter 5 of \cite{adams}). Using \eqref{trace} and \eqref{sobo}, we prove that $u_0$ approximates $u$ when $D$ has small volume.

\begin{lemma}\label{u0approx} 
For all $f \in H^{1/2} ( \partial \Omega)$, let $u$ and $u_0$ be the solutions to \eqref{mpde} and \eqref{hl}, respectively. Then, we have that 
$$\|{u- u_0}\|_{H^{1}(\Omega)} \leq C \, \epsilon^{\frac{d}{2}\left( 1- \frac{2}{p} \right)} \|{f}\|_{H^{1/2} (\partial \Omega)}$$ 
provided that $|D| = \mathcal{O}(\epsilon^{d})$ where $p \geq 2$ in $d=2$ and $2 \leq p \leq 6$ in $d=3$.
\end{lemma}
\begin{proof} Notice that, $u-u_0 \in H_{0}^{1} (\Omega)$ so we have that $\| u-u_0 \|_{H^1(\Omega)} \leq C  \| \nabla (u-u_0) \|_{L^2(\Omega)}$ by the Poincar\'{e} inequality. Therefore, appealing to Green's 1st Theorem as in the previous section  to obtain \eqref{vf} we have that 
\begin{align*}
 \int_{\Omega} \big | \nabla (u - u_0 ) \big |^2 \, \text{d}x  = - \int_{\partial D} \gamma u \overline{(u - u_0)} \, \text{d}s &\leq \gamma_{\text{max}} \|{u}\|_{L^{2}(\partial D)} \|{u - u_0}\|_{L^{2}(\partial D)}.
\end{align*} 
Now, by using the estimates in \eqref{trace} and \eqref{sobo}, we have that 
\begin{align*}
	 \|{u}\|_{L^{2}(\partial D)} \|{u - u_0}\|_{L^{2}(\partial D)} &\leq C  \|{u}\|^{1/2}_{L^{2}(D)} \|{u}\|^{1/2}_{H^{1}(D)} \|{u-u_0}\|^{1/2}_{L^{2}(D)} \|{u - u_0}\|^{1/2}_{H^{1}(D)} \\
	&\leq C \epsilon^{\frac{d}{2}\left( 1- \frac{2}{p} \right)} \|{u}\|_{H^{1} (\Omega)} \|{u - u_0}\|_{H^{1} (\Omega)}.
\end{align*}
This proves the claim by appealing to the well-posedness of \eqref{mpde}.
\end{proof}

From the above lemma, we have shown that $u$ can be approximated by $u_0$ in norm when $|D|$ is small. Under the same assumption, we will use the previous lemma along with \eqref{trace} and \eqref{sobo} to compare the magnitudes of the two integrals from equation \eqref{aexp}. We begin by analyzing the second integral and provide the following results.

\begin{lemma}\label{secondint} 
For $z \in \partial \Omega$ and $|D| = \mathcal{O}(\epsilon^{d} )$, we have that 
$$ \int_{\partial D} \gamma (x) (u - u_0 )(x) \partial_{\nu (z)} \mathbb{G} (x , z) \, \text{d}s(x) = \mathcal{O} \big( \epsilon^{d} \big) \quad \text{as} \quad \epsilon \to 0.$$ 
\end{lemma}
\begin{proof} 
In order to prove the claim, we must estimate 
\begin{align*}
	\left| \int_{\partial D} \gamma (x) (u - u_0 ) (x) \partial_{\nu (z)} \mathbb{G} (x , z) \, \text{d}s(x)  \right| &\leq C \|{u - u_0}\|_{L^{2}(D)}^{1/2} \|{u - u_0}\|_{H^1(D)}^{1/2} \|{\partial_{\nu (z)} \mathbb{G}(\cdot , z )}\|_{L^{2} ( \partial D)}\\
	&\leq C \epsilon^{\frac{d}{4}\left( 1- \frac{2}{p} \right)} \|{u - u_0}\|_{H^{1} (\Omega)} \|{\partial_{\nu (z)} \mathbb{G}(\cdot , z )}\|_{L^{2} ( \partial D)} \\
	&\leq C \epsilon^{\frac{3d}{4}\left( 1- \frac{2}{p} \right)} \|{f}\|_{H^{1/2}(\partial \Omega)} \|{\partial_{\nu (z)} \mathbb{G}(\cdot , z )}\|_{L^{2} ( \partial D)}
\end{align*} 
where we have used \eqref{trace}, \eqref{sobo}, and \lemref{u0approx} in order. We also have that 
\begin{align*}
	\|{\partial_{\nu (z)} \mathbb{G}(\cdot , z )}\|_{L^{2} ( \partial D)} \leq C \|{\partial_{\nu (z)} \mathbb{G}(\cdot , z )}\|_{H^1(D)}&\leq C \sqrt{ \epsilon^{d} \|{\mathbb{G}(\cdot ,z)}\|_{C^{1} (D)}^{2} + \epsilon^{d} \|{\mathbb{G}(\cdot , z) }\|_{C^{2}(D)}^{2}}\\
	&\leq C \epsilon^{d/2} \|{\mathbb{G}(\cdot , z )}\|_{C^{2} ( \Omega^{*})}
\end{align*}
where we have used \eqref{trace} for ${\partial_{\nu (z)} \mathbb{G}(\cdot , z )}$ on $\partial D$. The region $\Omega^{*}$ satisfies that $D \subset \Omega^{*} \subset \Omega$ for all $0 < \epsilon \ll 1$ with dist$(\partial \Omega , \Omega^{*})>0$. Thus, we have that for all $z \in \partial \Omega$ 
\begin{align}\label{powereps}
\left| \int_{\partial D} \gamma (x) (u - u_0 )(x) \partial_{\nu (z)} \mathbb{G} (x , z) \, \text{d}s(x) \right|  \leq C \epsilon^{\frac{d}{2} \left(\frac{5}{2} - \frac{3}{p} \right)} \|{f}\|_{H^{1/2}(\partial \Omega)}.
\end{align}
For $d=2$, we recall that $ p \geq 2$. In order to prove the claim, we impose the condition that 
$$2 = \frac{5}{2} - \frac{3}{p} \quad \text{making the exponent of $\epsilon$ equal to $d=2$ in \eqref{powereps}},$$ 
which yields that $p=6$. From the above inequality we get that  
$$ \int_{\partial D}\gamma (x) (u - u_0 )(x) \partial_{\nu (z)} \mathbb{G} (x , z)\, \text{d}s(x) \leq  C \epsilon^{2} \|{f}\|_{H^{1/2}(\partial \Omega)}.$$
Similarly, for $d=3$, we recall that $2 \leq p \leq 6$. 
Again, to prove the claim we impose that 
$$3 = \frac{3}{2} \left(\frac{5}{2} - \frac{3}{p} \right) \quad \text{again making the exponent of $\epsilon$ equal to $d=3$ in \eqref{powereps}},$$ 
which yields that $p=6$. Thus, we have that  
$$ \int_{\partial D} \gamma (x) (u - u_0 )(x) \partial_{\nu (z)} \mathbb{G} (x , z) \, \text{d}s(x)  \leq  C \epsilon^{3} \|{f}\|_{H^{1/2}(\partial \Omega)}.$$
Therefore, for both $d=2$ and $d=3$ taking $p=6$, we have that 
$$ \int_{\partial D} \gamma (x) (u - u_0 )(x) \partial_{\nu (z)} \mathbb{G} (x , z) \, \text{d}s(x)  = \mathcal{O} \big( \epsilon^{d} \big) \quad \text{as} \quad \epsilon \to 0$$
which proves the claim. 
\end{proof}

Next, we show that the first integral in \eqref{aexp} is order $\epsilon^{d-1}$. From this, equation \eqref{aexp} will imply that the first integral is the leading order term, rendering the second integral as negligible. The following lemma is key in deriving the asymptotic expansion.

\begin{lemma}\label{firstint} 
For all $z \in \partial \Omega$ where $D $ is given by \eqref{sball} we have that as $\epsilon \to 0$
$$ \int_{\partial D} \gamma (x) u_0(x) \partial_{\nu (z)} \mathbb{G}(x , z) \, \text{d}s(x) = \epsilon^{d-1} \sum_{j=1}^{J} | \partial B_j | \text{Avg}(\gamma_j) u_0 (x_j) \partial_{\nu (z)} \mathbb{G}(x_j , z) + \mathcal{O}(\epsilon^d)$$
where $\text{Avg}(\gamma_j)$ is the average value of $\gamma$ on $\partial D_j$. 
\end{lemma}
\begin{proof} 
By equation \eqref{sball}, we have that $x \in \partial D_j$ if and only if $x = x_j + \epsilon y$ for some $y \in \partial B_j$. Now, recall that both $u_0$ and  $\partial_{\nu (z)} \mathbb{G}(\cdot \, , z)$ are smooth in the interior of $\Omega$ since $z \in  \partial \Omega$. Therefore, we have that for all $x \in \partial D_j$ as $\epsilon \to 0$
$$u_0 ( x) \partial_{\nu (z)} \mathbb{G}( x , z) = u_0 ( x_j + \epsilon y) \partial_{\nu (z)} \mathbb{G}( x_j + \epsilon y , z)= u_0(x_j)  \partial_{\nu (z)} \mathbb{G}(x_j , z) + \mathcal{O}(\epsilon)$$ 
by appealing to Taylor's Theorem. From this, we obtain that 
\begin{align*}
	\int_{\partial D} \gamma (x) u_0(x) \partial_{\nu (z)} \mathbb{G}(x , z) \, \text{d}s(x)   &=  \sum_{j=1}^{J} \int_{\partial D_j} \gamma (x) u_0 (x_j + \epsilon y) \partial_{\nu (z)} \mathbb{G}(x_j + \epsilon y , z) \, \text{d}s(x)\\
	&=  \sum_{j=1}^{J}  \left( u_0(x_j)  \partial_{\nu (z)} \mathbb{G}(x_j , z) + \mathcal{O}(\epsilon) \right)\int_{\partial D_j} \gamma (x)  \, \text{d}s(x).
\end{align*} 
Therefore, we have that 
$$ \int_{\partial D} \gamma (x) u_0(x) \partial_{\nu (z)} \mathbb{G}(x , z) \, \text{d}s(x) = \epsilon^{d-1} \sum_{j=1}^{J} | \partial B_j | \text{Avg}(\gamma_j)  u_0 (x_j) \partial_{\nu (z)} \mathbb{G}(x_j , z) + \mathcal{O}(\epsilon^d)$$
as $\epsilon \to 0$ where $\text{Avg}(\gamma_j)$ denotes the average value of $\gamma$ on $\partial D_j$ as well as using the fact that $| \partial D_j |  = \epsilon^{d-1} | \partial B_j |$. 
\end{proof}

Using \lemref{secondint} and \lemref{firstint}, it is clear that for a specified $z \in \partial \Omega$, the current gap is dominated by the first integral from equation \eqref{aexp}. Therefore, we have proven an asymptotic expansion for the current gap operator. Similar results have been proven in \cite{MUSIC-ammari-eit,MUSIC-Hanke} using boundary integral operators. 

\begin{theorem}\label{mainasy} 
For any $z \in \partial \Omega$ we have that 
$$ (\Lambda - \Lambda_0) f(z) = - \epsilon^{d-1} \sum_{j=1}^{J} | \partial B_j | \text{Avg}(\gamma_j)  u_0 (x_j) \partial_{\nu (z)} \mathbb{G}(x_j , z)+ \mathcal{O}(\epsilon^d) \quad \text{as} \quad \epsilon \to 0$$ 
provided that $D$ is given by \eqref{sball}.
\end{theorem}

We use this approximation to develop an algorithm that detects the centers of the defective regions. Consequently, this will allow us to recover the region $D$.

We now, study the MUSIC algorithm which can be considered as a discrete analogue of the factorization method (see for e.g. \cite{cheney1,kirschbook,MUSIC-kirsch}). In particular, we connect the centers of the defective regions given by $\{x_j : j = 1, \hdots , J\}$ to a matrix denoted by \textbf{F} that is defined using physical measurements on $\partial \Omega$. We assume that $\Omega$ is the unit ball for $d=2$ and that we have a finite number of data $N+1$ on $\partial \Omega$ where $ J < N +1$. In order to proceed we must first define 
the sesquilinear dual-product 
\begin{equation}\label{dualprod}
\langle \varphi , \psi \rangle_{\partial \Omega} = \int_{\partial \Omega} \varphi \, \overline{\psi} \, \text{d}s \quad \text{for all} \quad \varphi \in H^{1/2} ( \partial \Omega ) \quad \text{and} \quad \psi \in H^{-1/2} ( \partial \Omega )
\end{equation}
between the Hilbert Space $H^{1/2}(\partial \Omega )$ and its dual space $H^{-1/2} ( \partial \Omega )$ where $L^2 ( \partial \Omega)$ is the Hilbert pivot space. Recall, that we have the following
$$ H^{1/2} ( \partial \Omega ) \subset L^2 ( \partial \Omega) \subset H^{-1/2} ( \partial \Omega )$$ 
with dense inclusions. Physically, this dual-product relates the voltage and the induced current on $\partial \Omega$ and is used to construct the matrix \textbf{F}. The dual-product will also be used in the upcoming sections. Using \thmref{mainasy}, for any $g,f \in H^{1/2} (\partial \Omega)$ we have that
\begin{align*}
\big\langle g, \overline{(\Lambda - \Lambda_0 ) f} \big\rangle_{\partial \Omega}   &=  \bigg\langle g, - \epsilon^{d-1} \sum_{j=1}^{J}  | \partial B_j | \text{Avg}(\gamma_j) \overline{u_{0} (x_j,f)} \partial_{\nu (z)} \mathbb{G}(x_j , z) + \mathcal{O}(\epsilon^d)  \bigg\rangle_{\partial \Omega} \\
	&=  - \epsilon^{d-1} \sum_{j=1}^{J} | \partial B_j | \text{Avg}(\gamma_j) u_{0} (x_j , f) \big\langle g , \partial_{\nu (z)} \mathbb{G}( x_j , z) \big\rangle_{\partial \Omega}  + \mathcal{O}(\epsilon^d)
\end{align*} 
where $u_{0}( \cdot , f)$ is the solution to \eqref{hl} with boundary condition $f$. Since $z \in \partial \Omega$, we have that
\begin{align*}
	 \big\langle g , \partial_{\nu (z)} \mathbb{G}(x_j , z ) \big\rangle_{\partial \Omega}   &=   \int_{ \partial \Omega} u_{0}(z , g)  \partial_{\nu (z)} \mathbb{G}(x_j , z ) \, \text{d}s(z) = - u_{0} (x_j , g)
\end{align*}
where $u_{0} (\cdot \,  , g)$ is the solution to \eqref{hl} with boundary condition $g$ (see for e.g. Chapter 2 of \cite{evans}). Therefore, we have that as $\epsilon \to 0$
\begin{equation}\label{voltage_current}
\big\langle g, \overline{(\Lambda - \Lambda_0 ) f} \big\rangle_{\partial \Omega} = \epsilon^{d-1} \sum_{j=1}^{J}| \partial B_j | \text{Avg} (\gamma_j) u_{0} (x_j , g) u_{0} (x_j , f)  + \mathcal{O}(\epsilon^d). 
\end{equation}
We now let $g = \text{e}^{\text{i} m \theta}$ and $f = \text{e}^{\text{i}n \theta}$ for $m,n = 0, \hdots , N$ where $\theta$ is the angle formed by a points on $\partial \Omega$ when converted to polar coordinates. Using only the leading order term of \eqref{voltage_current}, we define the matrix 
$$ \textbf{F}_{n,m} = \epsilon^{d-1} \sum_{j=1}^{J} | \partial B_{j} | \text{Avg}(\gamma_{j}) u_{0} (x_j , f_m) u_{0} (x_j , f_n). $$ 
We factorize \textbf{F} by defining matrices \textbf{U} $\in \mathbb{C}^{(N+1) \times J}$ and \textbf{T} $\in \mathbb{C}^{J \times J}$, where the matrices \textbf{U} and \textbf{T} are given by 
$$ \text{\textbf{U}}_{m,j} = u_0 (x_j , f_m) \quad \text{ and } \quad \text{\textbf{T}} = \text{diag} \big(\epsilon^{d-1} | \partial {B_j}| \text{Avg}(\gamma_{j}) \big).$$  
From the definition of the matrices, we have that \textbf{F} = \textbf{UTU}$^{\top}$. Notice, that all the diagonal entries in $\textbf{T}$ are non-zero. We now define the vector $\boldsymbol{\phi}_x \in \mathbb{C}^{N+1}$ for any point $x \in \mathbb{R}^d$ by 
\begin{equation}\label{phix}
\boldsymbol{\phi}_x  = \big( u_{0} (x , f_{0}) , \hdots , u_{0} (x , f_{N} )  \big)^{\top}.
\end{equation}
The ultimate goal of this section is to prove that $\boldsymbol{\phi}_x $ is in the range of $\text{\textbf{FF}}^{*}$ if and only if $x \in \brac{x_j : j = 1 , \hdots , J}$. This is a discrete reformulation of the result of the factorization method presented in \cite{cheney1,MUSIC-kirsch}. We are interested in reconstructing regions $D_j$, so it is sufficient to prove the result only for values $x \in \Omega$. We now state a result that can be proven by using standard arguments from Linear Algebra (see for e.g. \cite{MUSIC-elastic}).

\begin{lemma} Let the matrix \textbf{F} have the following factorization \textbf{F} = \textbf{UTU}$^{\top}$ where $ \textbf{U} \in \mathbb{C}^{(N+1) \times J}$ and $ \textbf{T} \in \mathbb{C}^{J \times J}$ with $N+1 > J$. Assume that 
the matrix \textbf{U} has full rank $J$ and the matrix \textbf{T} is invertible. Then Range$(\textbf{U})$= $Range$(\textbf{FF}$^{*}$).
\end{lemma}

We now construct an indicator function derived from the previous Lemma to determine the location of the defective regions. For each sampling point $x \in \Omega$ we will show that $\boldsymbol{\phi}_x $ is in the range of \textbf{FF}$^{*}$ if and only if $x \in \{ x_j : j = 1 , \hdots , J\}$. We introduce an auxiliary result that connects the location of the unknown regions to the range of the matrix \textbf{U}.

\begin{theorem}
Assume that $N+1 > J$. Then, we have 
that the matrix \textbf{U} has full rank and 
$\boldsymbol{\phi}_x \in Range ( \text{\textbf{U}} )$ if and only if $x \in \brac{x_j : j = 1 , \hdots , J}$ where $\boldsymbol{\phi}_x$ is defined as in \eqref{phix}.
\end{theorem}
\begin{proof} 
It is clear that $\boldsymbol{\phi}_{x_j}$ is in the range of \textbf{U} since $\boldsymbol{\phi}_{x_j}$ is the $j$-th column of \textbf{U}. Conversely, suppose $x \notin \{x_j : j = 1, \hdots , J\}$ and by way of  contradiction assume that $ \boldsymbol{\phi}_x  \in$ Range(\textbf{U}). This would imply that there exists $\boldsymbol{\alpha} \in \mathbb{C}^{J}$ such that 
$$ z^n - \sum_{j=1}^{J} \alpha_j z_{j}^{n} = 0 \quad \text{for all } \quad n = 0 , 1, \hdots , N$$
where $ z^n = u_0 (x , f_n ) = |x|^n \text{e}^{\text{i}n \theta}$.  
Since we have assumed that $N+1>J$, this would imply that the square Vandermonde matrix denoted by ${\bf V}(z,z_1 ,\cdots , z_{J})$ satisfies that the non-zero vector $(1,-\boldsymbol{\alpha})^\top$ is in its null space. This is a contradiction due to the fact that Det$\big({\bf V}(z,z_1 ,\hdots , z_{J})\big) \neq 0$ since $z \neq z_j$ and $z_i \neq z_j$ for all $i \neq j$. Moreover, the fact that \textbf{U} has full rank is a consequence of the above argument.
\end{proof}

Combining the two previous results, we have a MUSIC algorithm to recover the centers of the defective regions $\{x_j : j = 1 , \hdots ,J\}$ from the physical measurements. 

\begin{theorem}\label{musicthm}
Assume that $N+1 > J$. Then for all $x \in \Omega$ 
$$\boldsymbol{\phi}_x \in Range ( \text{\textbf{FF}}^{*} ) \quad \text{if and only if} \quad x \in \brac{x_j : j = 1 , \hdots , J}$$ 
where $\boldsymbol{\phi}_x$ is defined as in \eqref{phix}.
\end{theorem}

Notice, that the matrix \textbf{F} can be approximated by the known current gap operator $\Lambda-\Lambda_0$. This implies that Theorem \ref{musicthm} can be used to recover the centers of the defective regions $\{x_j : j = 1 , \hdots ,J\}$. To this end, we must verify whether $\boldsymbol{\phi}_x \in$ Range$(\text{\textbf{FF}}^{*})$. This is equivalent to $\text{\textbf{P}}\boldsymbol{\phi}_x$ = 0 where $\text{\textbf{P}}$ is the orthogonal projection onto the Null$(\text{\textbf{FF}}^{*})$.

\subsection{Numerical Validation for the MUSIC Algorithm }
We now provide some numerical examples of recovering locations $\{x_j : j = 1 , \hdots ,J\}$ using Theorem \ref{musicthm}. All of our numerical experiments are done with the software \texttt{MATLAB} 2020a. We will let $\Omega$ be given by the unit ball in $\R^2$ and we need to compute the current gap operator $\Lambda-\Lambda_0$ applied to $f = \text{e}^{\text{i}n \theta}$. Lemmas \ref{secondint} and \ref{firstint} imply that 
$$(\Lambda-\Lambda_0)f (z) \approx - \int_{\partial D} \gamma(x) u_0(x ,f) \partial_{\nu (z)} \mathbb{G} (x , z) \, \text{d}s(x).$$
This can be seen as an analog to the Born approximation used in scattering theory (see for e.g. \cite{MUSIC-kirsch}). Therefore, we can compute $(\Lambda-\Lambda_0)f$ using the `\texttt{integral}' command in \texttt{MATLAB}. 

Since $f = \text{e}^{\text{i}n \theta}$ it is clear that the harmonic lifting is given by $u_0(x ,f)=|x|^n \text{e}^{\text{i}n \theta}$.  It is also well known that the normal derivative of $ \mathbb{G} (x , z)$ is given by 
$$ \partial_{\nu(z)} \mathbb{G}\big ( x  , z \big ) \big|_{\partial \Omega} =  \frac{1}{2 \pi} \left[ \frac{1 - |x|^2 }{|x|^2 + 1 - 2 |x| \text{cos}(\theta - \theta_{z})}   \right]$$
for $z \in \partial \Omega$. We can then easily compute $(\Lambda-\Lambda_0) \text{e}^{\text{i}n \theta}$ for $n=0, \hdots , 20$. Here we evaluate current gap for 64 equally spaced points on the unit circle. By appealing to the asymptotic result in \eqref{voltage_current}, we have that 
$$\textbf{F}_{n,m} \approx \big\langle  \text{e}^{\text{i}m\theta}, \overline{(\Lambda - \Lambda_0 ) \text{e}^{\text{i}n\theta}} \big\rangle_{\partial \Omega}\quad \text{ such that } \quad \textbf{F} \in \C^{21 \times 21}$$
which is approximated via a Riemann sum using the `\texttt{dot}' command in \texttt{MATLAB}.

Once $\textbf{F}$ has been approximated we can use Theorem \ref{musicthm} to recover the locations of the components of $D$. We only need to check if the vector $\boldsymbol{\phi}_x$ is in the range of $\text{\textbf{FF}}^{*}$. Therefore, we compute the norm 
$$\| \text{\textbf{P}}\boldsymbol{\phi}_x \|^2_2 = \sum\limits_{\ell=r+1}^{21} \left|\big( \boldsymbol{\phi}_x , {\bf u}_\ell \big) \right|^2 $$
where the vectors ${\bf u}_\ell$ are the orthonormal eigenvectors for the matrix ${\bf F}{\bf F}^*$ and  $r=$Rank$\big( {\bf F}{\bf F}^* \big)$. 
Recall, that the vector 
$$\boldsymbol{\phi}_x = \bigg(1,  |x| \text{e}^{\text{i} \theta}\,  , \, \hdots \, , \, |x|^{20} \text{e}^{20 \text{i}\theta} \bigg)^\top$$ 
by equation \eqref{phix} where $\theta$ is the polar angle for the sampling point $x \in \Omega$. Here $\text{\textbf{P}}$ denotes the orthogonal projection onto the Null$(\text{\textbf{FF}}^{*})$. The imaging functional is given by 
$$W_{\text{MUSIC}}(x) = \left[ \sum\limits_{\ell=r+1}^{21} \left|\big( \boldsymbol{\phi}_x , {\bf u}_\ell \big) \right|^2 \right]^{-1} \quad \text{ for any  } \quad x \in \Omega$$
which satisfies that $W_{\text{MUSIC}}(x)\gg1$ for $x=x_j$ and $W_{\text{MUSIC}}(x)= \mathcal{O}(1)$ for $x \neq x_j$.

In Figures \ref{recon-music1} and \ref{recon-music2} we use the imaging functional $W_{\text{MUSIC}}(x)$ given above to recover the locations of the two components of the region $D$. In theses experiments, the region 
$$D =\big( x_1 + \epsilon B(0,1) \big) \bigcup  \big( x_2 + \epsilon B(0,1)\big)$$
with $B(0,1)$ being the unit ball centered at the origin. The points $x_1$ and $x_2$ are points contained in the region $\Omega$. We will take the transmission parameter to be given by $\gamma=1$ on the boundary of both components of $D$. Here, we take $\epsilon = 0.01$ as well as adding $1\%$ random noise to the computed current gap to simulate error in measured data. \\

\noindent{\bf Example 1: }\\
In our first example presented here we let 
$$x_1 = (-0.25 , -0.25) \quad \text{ and } \quad x_2 = (0.25 , 0.25)$$ 
for the reconstruction in Figure \ref{recon-music1}. Presented is a contour and surface plot of the imaging functional $W_{\text{MUSIC}}(x)$. As we can see from the data tips, the imaging functional has spikes at the points 
$$\widetilde{x}_1 = (-0.2323 , -0.2323) \quad \text{ and } \quad \widetilde{x}_2 = (0.2323 , 0.2323).$$ 
Here we see that the locations $\widetilde{x}_1$ and $\widetilde{x}_2$ given by the MUSIC Algorithm provide an approximation for the actual locations of the components of the region $D$. \\
\begin{figure}[ht]
\centering 
\includegraphics[scale=0.38]{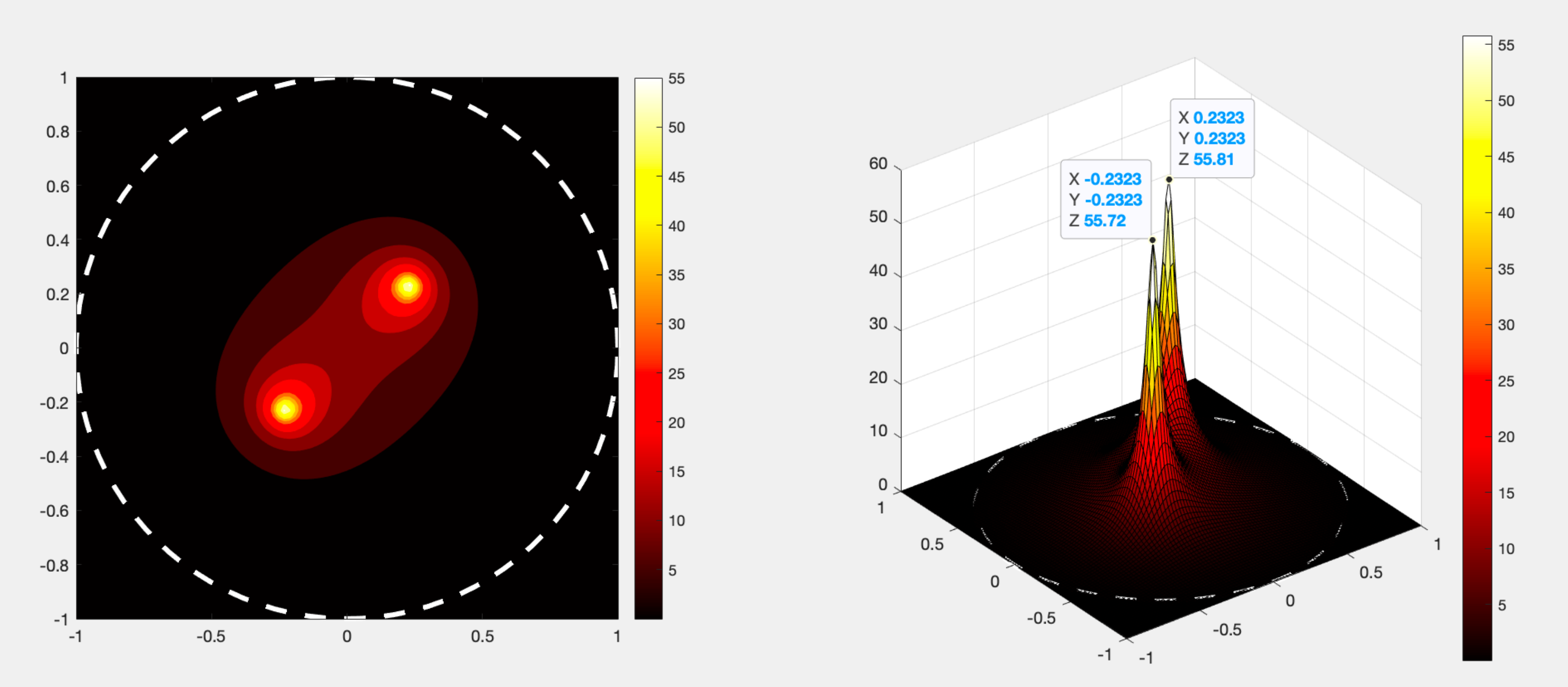}
\caption{Reconstruction of the locations $x_1 = (-0.25 , -0.25)$ and $x_2 = (0.25 , 0.25)$ via the MUSIC Algorithm. Contour plot  on the left and Surface plot on the right for $W_{\text{MUSIC}}(x)$.}
\label{recon-music1}
\end{figure}

\noindent{\bf Example 2:  }\\
Now, in this example presented here we let 
$$x_1 = (-0.25 , 0.25) \quad \text{ and } \quad x_2 = (-0.25 , -0.25)$$ 
for the reconstruction in Figure \ref{recon-music2}. Presented is a contour and surface plot of the imaging functional $W_{\text{MUSIC}}(x)$. As we can see from the data tips, the imaging functional has spikes at the points 
$$\widetilde{x}_1 = (-0.2929 , 0.2525) \quad \text{ and } \quad \widetilde{x}_2 =  (-0.2929 , -0.2727).$$ 
Again, in this example we see that the locations $\widetilde{x}_1$ and $\widetilde{x}_2$ provide an approximation for the  locations of the components of the region $D$. 
\begin{figure}[ht]
\centering 
\includegraphics[scale=0.38]{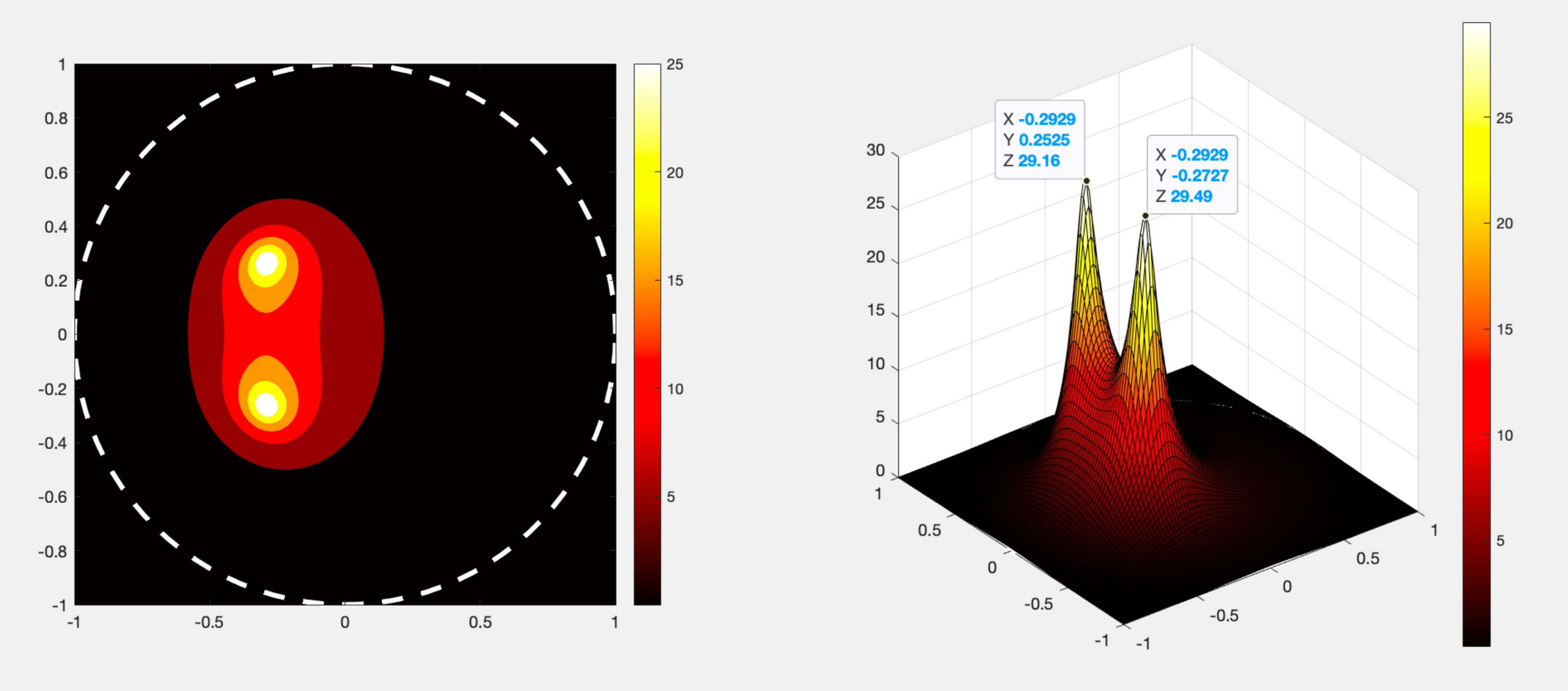}
\caption{Reconstruction of the locations $x_1 = (-0.25 , 0.25)$ and $x_2 = (-0.25 , -0.25)$ via the MUSIC Algorithm. Contour plot on the left and Surface plot on the right for $W_{\text{MUSIC}}(x)$.}
\label{recon-music2}
\end{figure}

\section{\textbf{Recovering Extended Regions}} \label{Factorization section}
In this section, we focus on the case of an extended region $D$. Therefore, the asymptotics developed in the previous section is invalid and we must employ a different technique for recovering the region of interest. The theory used here was developed in \cite{harris1} and will allow us to derive a different imaging functional for extended regions. The analysis is based on the factorization of the current gap operator $(\Lambda - \Lambda_0 )$. The goal is to again derive an imaging functional using the spectral decomposition (or singular value decomposition) of the known current gap operator.

\subsection{\textbf{Regularized Factorization Method}}\label{RegFM}
In this section, we employ the regularized factorization method developed in \cite{harris1} and provide a different approach to solve the inverse shape problem via another sampling method. In general, sampling algorithms connect the support of the unknown region to an indicator function deriving from an ill-posed equation involving the measurements operator and a singular solution to the background problem. We again, focus on creating an algorithm for recovering the unknown region $D$ from the measurements operator given by the current gap operator $(\Lambda - \Lambda_0 )$. To this end, we will focus on recovering extended defective region(s).

Inspired by the current gap operator $(\Lambda - \Lambda_0 )$, we note that $u - u_0 \in H^1_0(\Omega)$ solves 
$$ - \Delta (u-u_0 ) = 0 \quad \text{in} \quad \Omega \textbackslash \partial D \quad \text{with} \quad  [\![\partial_\nu (u - u_0) ]\!] \big|_{\partial D} = \gamma u\big|_{\partial D} .$$
So, we define $w \in H^{1}_0 (\Omega )$ to be the unique solution of 
\begin{equation}\label{w}
- \Delta w = 0 \quad \text{in} \quad \Omega \textbackslash \partial D \quad \text{with} \quad [\![\partial_\nu w ]\!]\big|_{\partial D} = \gamma h
\end{equation}
for any given $h \in L^{2} ( \partial D)$. One can show that \eqref{w} is well-posed by appealing to a variational formulation argument as in Section \ref{dp-ip}. Therefore, we can define the bounded linear Source-to-Neumann operator 
$$ G : L^{2} ( \partial D ) \rightarrow H^{-1/2} (\partial \Omega ) \quad \text{given by} \quad Gh = \partial_{\nu} w \big|_{\partial \Omega}$$ 
where $w$ is the unique solution of \eqref{w}. 
The following observation  allows us to further understand the connection between operators $G$ and $(\Lambda-\Lambda_0)$. 
By well-posedness of \eqref{w}, we have that 
$$\partial_{\nu} w \big \rvert_{\partial \Omega} = ( \Lambda - \Lambda_0 ) f \quad \text{ provided that } \quad h = u \big|_{\partial D}.$$ 
From this, we define the solution operator for \eqref{mpde} as
$$ S : H^{1/2} ( \partial \Omega ) \rightarrow L^{2} (\partial D ) \quad \text{given by} \quad Sf = u \big|_{\partial D}$$
Thus, we see that $(\Lambda - \Lambda_0 )f = GSf$ for any $f \in H^{1/2} ( \partial \Omega)$. In order to gain more information, we need to factorize $( \Lambda - \Lambda_0 ) $ further by decomposing $G$. This requires analyzing the adjoint of the operator $S$. The following result defines the adjoint of $S$.

\begin{theorem} \label{adjoint}
The adjoint operator $S^{*}: L^2 ( \partial D) \rightarrow H^{-1/2} ( \partial \Omega )$ is given by $S^{*}g = \partial_{\nu} v \big|_{\partial \Omega}$ where $v \in H^1_0 (\Omega)$ satisfies 
\begin{equation} \label{adv}
- \Delta v = 0 \quad \text{in} \quad  \Omega \textbackslash \partial D \quad \text{with} \quad [\![\partial_\nu v ]\!] \big|_{\partial D} = \gamma v\big|_{\partial D}  + g
\end{equation} 
Moreover, the operator $S$ is compact and injective.
\end{theorem}
\begin{proof} 
Notice, that by using a variational argument we can establish that the solution $v \in H_{0}^{1} ( \Omega )$ exists, is unique, and continuously depends on $g \in L^{2} ( \partial D)$. Using a similar technique used to derive \eqref{vf} and Green's 2nd Theorem, we have that 
$$ 0 =  \int_{\partial \Omega} \overline{v} \, \partial_{\nu}u - f \, \partial_{\nu} \overline{v} \, \text{d}s + \int_{\partial D} \overline{v} ( \partial_{\nu} u^{-} - \partial_{\nu} u^{+}) \, \text{d}s + \int_{\partial D} u ( \partial_{\nu} \overline{v}^{+} - \partial_{\nu} \overline{v}^{-}) \, \text{d}s. $$
By the  boundary condition on $\partial D$ for $u$, this reduces to 
$$ \int_{\partial \Omega} f \, \partial_{\nu} \overline{v} \, \text{d}s = \int_{\partial D} \big( [\![\partial_\nu \overline{v} ]\!] - \gamma \overline{v} \big) u \, \text{d}s.$$ 
Using boundary condition on $\partial D$ for $v$, we have that $$ \int_{\partial D} \big( [\![\partial_\nu \overline{v} ]\!] - \gamma \overline{v} \big) u \, \text{d}s = \int_{\partial D} u \overline{g} \, \text{d}s.$$ 
Thus, we have that $$ (Sf , g)_{L^{2} ( \partial D)} = \int_{\partial D} u \overline{g} \, \text{d}s = \int_{\partial \Omega} f \partial_{\nu} \overline{v} \, \text{d}s = \langle f , S^{*} g \rangle_{\partial \Omega}$$
for all $f \in H^{1/2} (\partial \Omega)$ and $g \in L^{2} (\partial D)$ which implies that $S^{*} g = \partial_{\nu} v \big|_{\partial D}$.

To prove injectivity, we let $Sf = 0$ which implies that $u =0$ in $\bar{D}$. Using Holmgren's Theorem (see for e.g. \cite{holmgren}), we have that $u = 0$ in $\Omega$. Then by the Trace Theorem, we have that $f = 0$ on $\partial \Omega$, proving that $S$ is injective. Furthermore, the compact embedding of $H^{1/2} (\partial D)$ into $L^{2} ( \partial D)$ implies that $S$ is compact.
\end{proof}

In order to complete the factorization of the current gap operator, we need to define a middle operator $T$. Recall, that $w$ is the unique solution to equation \eqref{w}, which implies that $w$ is harmonic in $\Omega \textbackslash \partial D$  and 
$$ [\![\partial_\nu w ]\!] \big|_{\partial D} = \gamma w\big|_{\partial D} +  \gamma \big[ h - w\big|_{\partial D} \big]$$
by appealing to the Robin transmission condition. Therefore, we have that 
$$\partial_{\nu} w \big|_{\partial \Omega} = Gh \quad \text{as well as} \quad  \partial_{\nu} w \big|_{\partial \Omega} =S^{*} \gamma \big[ h - w \big|_{\partial D} \big]$$
by the well-posedness of \eqref{adv} and Theorem \ref{adjoint}. Motivated by this, we define the operator 
$$T: L^2 ( \partial D ) \rightarrow L^2 ( \partial D ) \quad \text{given by} \quad Th =  \gamma \big[ h - w|_{\partial D} \big].$$
By the well-posedness of \eqref{w}, $T$ is a bounded linear operator. Recall, that we had already established that $(\Lambda - \Lambda_0) = GS$ and observe that we have factorized the operator $G$ as $G = S^{*}T$. This gives the following result.

\begin{theorem}\label{factorization} 
The difference of the DtN mappings $(\Lambda - \Lambda_0): H^{1/2} ( \partial \Omega) \rightarrow H^{-1/2} ( \partial \Omega )$ has the symmetric  factorization $( \Lambda - \Lambda_0 ) = S^{*} TS$. 
\end{theorem}

In order to apply Theorem 2.3 from \cite{harris1} to solve the inverse problem of recovering $D$ from the current gap operator $(\Lambda - \Lambda_0)$, we need to prove that $T$ is coercive and also characterize the region $D$ by the range of $S^{*}$. Satisfying these remaining conditions would allow us to reconstruct $D$ from the measure data $\Lambda f$ and computable Neumann data $\Lambda_0 f$ on the known exterior boundary. The following two results will allow us to prove some useful properties of the current gap operator using the symmetric factorization from the previous theorem. We now prove the coercivity of the operator $T$.

\begin{theorem}\label{tcoercive}
The operator $T: L^2 ( \partial D ) \rightarrow L^2 ( \partial D )$ given by $Th =\gamma \big[ h - w \big|_{\partial D} \big]$ is coercive on $L^2 ( \partial D )$, where $h \in L^{2}( \partial D)$ and $w \in H^{1}_{0} (\Omega)$ satisfies \eqref{w}.
\end{theorem}
\begin{proof} Using the Robin transmission condition on $\partial D$ in equation \eqref{w}, we have that $$ (Th, h)_{L^2 ( \partial D) } = \int_{\partial D} \gamma (h - w) \overline{h} \, \text{d}s =  \int_{\partial D} \gamma |h|^2 \, \text{d}s - \int_{\partial D} w \, [\![\partial_\nu \overline{w} ]\!] \, \text{d}s. $$
Following a similar technique used to derive \eqref{vf}, we have that 
$$ \int_{\Omega \backslash D} |\nabla w|^2 \, \text{d}x = - \int_{\partial D} w  \partial_{\nu} \overline{w}^{+} \, \text{d}s \quad \text{and} \quad \int_{D} |\nabla w|^2 \, \text{d}x = \int_{\partial D} w \partial_{\nu} \overline{w}^{-} \, \text{d}s.$$
Adding both equations above and using the boundary condition on $\partial D$  yields
 $$ \int_{\Omega} |\nabla w|^2 \, \text{d}x = - \int_{\partial D} w [\![\partial_\nu \overline{w} ]\!] \, \text{d}s.$$ Therefore, we have that 
 $$(Th, h)_{L^2 ( \partial D) } = \int_{\partial D} \gamma |h|^2 \, \text{d}s + \int_{\Omega} |\nabla w |^{2} \, \text{d}x \geq \gamma_{\text{min}} \int_{\partial D}  |h|^2 \, \text{d}s$$ 
which proves the claim.
\end{proof}

These follow two results allow us to prove the main theorem of this section which characterizes the analytical properties of the current gap operator.
 
\begin{theorem}\label{DtNprop}
The difference of the DtN mappings $(\Lambda - \Lambda_0 ): H^{1/2} ( \partial \Omega ) \rightarrow H^{-1/2} ( \partial \Omega )$ is compact, injective, and has dense range.
\end{theorem}

\begin{proof} 
The compactness is a consequence of Theorem \ref{adjoint} and \ref{factorization}, since $S$ is compact. We prove that the current gap operator $(\Lambda - \Lambda_0)$ is injective with a dense range, using the same argument. More specifically, we show that the set of annihilators for Range$(\Lambda - \Lambda_0)$ and Null$(\Lambda - \Lambda_0 )$ are trivial. To this end, note that for all $f,g \in H^{1/2} ( \partial \Omega)$
\begin{align*}
 \langle g , (\Lambda - \Lambda_0) f \rangle_{\partial \Omega} &=  \int_{\partial \Omega} g \, \partial_{\nu} \overline{u (\cdot \, , f)} - g \, \partial_{\nu} \overline{u_{0} (\cdot \,, f)} \, \text{d}s\\
	&= \int_{\partial \Omega} u (\cdot \,, g) \, \partial_{\nu} \overline{u (\cdot \,, f) } - u_{0}(\cdot \,, g) \, \partial_{\nu} \overline{u_{0}(\cdot \, , f)} \, \text{d}s 
\end{align*} 
where the pairs $(u(\cdot , f) , u(\cdot , g))$ and $(u_{0}(\cdot , f) , u_{0}(\cdot , g))$ are solutions to \eqref{mpde} and \eqref{hl} using boundary conditions $f$ and $g$, respectively. Appealing to Green's 1st Theorem we obtain 
\begin{align*}
\langle g , (\Lambda - \Lambda_0) f \rangle_{\partial \Omega} &= \int_{\Omega} \nabla u(\cdot \, , g) \cdot \nabla \overline{u(\cdot \, ,f)} \, \text{d}x -  \int_{\Omega} \nabla u_{0}(\cdot \, , g) \cdot \nabla \overline{u_{0}(\cdot \, , f)} \, \text{d}x \\
& \hspace{1in} +  \int_{\partial D} \gamma \, u(\cdot \, ,g) \, \overline{u(\cdot \, , f)} \, \text{d}s
\end{align*}  
by using equations \eqref{mpde} and \eqref{hl}. Now suppose $f \in H^{1/2} ( \Omega)$ is an annihilator for Range$(\Lambda - \Lambda_0)$ or $f \in \text{Null} (\Lambda - \Lambda_0 )$. In either case, we have that
\begin{align*}
	0   &=  \langle f , (\Lambda - \Lambda_0 ) f \rangle_{\partial \Omega} \\
	&= \int_{\Omega} | \nabla u(\cdot \, , f) |^{2} \, \text{d}x - \int_{\Omega} | \nabla u_{0}(\cdot \, , f) |^{2} \, \text{d}x + \int_{\partial D} \gamma | u(\cdot \, , f) |^{2} \, \text{d}x \\
	&\geq \int_{\partial D} \gamma | u(\cdot \, , f) |^{2} \, \text{d}x
\end{align*} 
where we have used that $u_{0}(\cdot , f)$ satisfying equation \eqref{hl} minimizes the Dirichlet energy. By \thmref{adjoint}, $S$ is injective which implies that $f=0$, proving both claims.
\end{proof}

All of the theorems of this section imply that the current gap operator $\Lambda - \Lambda_0$ satisfies all of the conditions of Theorem 2.3 of \cite{harris1}. That is, $$ \ell \in Range(S^{*}) \quad \text{if and only if} \quad \liminf_{\alpha \rightarrow 0} \langle f_{\alpha} , (\Lambda - \Lambda_0 ) f_{\alpha} \rangle_{\partial \Omega} < \infty$$ where $f_{\alpha}$ is the regularized solution to $(\Lambda - \Lambda_0 ) f = \ell$. Since $\Lambda - \Lambda_0$ is compact and injective with a dense range, we can apply any regularization scheme such as Tikhonov or Spectral cut-off. However, we must still connect the domain $D$ to the range of the operator $S^{*}$. To this end, we once again use the Dirichlet Green's function for the negative Laplacian for the known domain $\Omega$, $\mathbb{G}(\cdot , z) \in H_{loc}^{1} (\Omega \textbackslash \brac{z})$ for any fixed $z \in \Omega$. The idea of the following result is to show that due to the singularity at $z$, the normal derivative of the Green's function is not contained in the range of $S^{*}$ unless the singularity is contained within the region of interest $D$.

\begin{theorem} \label{greenchar}
The operator $S^{*}$ is such that for any $z \in \Omega$ 
$$\partial_{\nu} \mathbb{G} (\cdot , z) \big|_{\partial D} \in Range(S^{*}) \quad \text{if and only if} \quad z \in D.$$
\end{theorem}

\begin{proof} 
To prove the claim, assume $z \in \Omega \textbackslash \overline{D}$. Suppose by contradiction that there exists $g_z \in L^2 ( \partial D )$ such that $S^{*} g_z = \partial_{\nu} \mathbb{G} ( \cdot , z) \big \rvert_{\partial \Omega}$. This implies that $\exists \, v_z \in H^{1} ( \Omega )$ such that 
$$ - \Delta v_z = 0 \quad \text{in} \quad \Omega \textbackslash \partial D \quad \text{with} \quad v_z \big \rvert_{\partial \Omega} = 0 \quad \text{and} \quad [\![\partial_\nu v_z ]\!] \big|_{\partial D} = \gamma v_z \big|_{\partial D} + g_z . $$ 
Furthermore, $\partial_{\nu} v_z \big|_{\partial \Omega} = \partial_{\nu} \mathbb{G}(\cdot \, , z ) \big|_{\partial \Omega}$ and we have that $v_z$ satisfies 
$$ - \Delta v_z = 0 \quad \text{in} \quad \Omega \textbackslash{D} \quad \text{with} \quad v_z \big|_{\partial \Omega} = 0 \quad \text{and} \quad \partial_{\nu} v_z \big|_{\partial \Omega} = \partial_{\nu} \mathbb{G}(\cdot \, , z ) \big|_{\partial \Omega} .$$ 
So we define $W_z = v_z - \mathbb{G}(\cdot , z )$ and note that 
$$ - \Delta W_z = 0 \quad \text{in} \quad \Omega \textbackslash ( \overline{D} \cup \{z\} ) \quad \text{with} \quad W_z \big \rvert_{\partial \Omega} = 0 \quad \text{and} \quad \partial_{\nu} W_z \big|_{\partial \Omega} = 0.$$ 
By Holmgen's Theorem \cite{holmgren}, we conclude that $W_z = 0$ in $\Omega \textbackslash (\overline{D} \cup \{z\})$. That is, $v_z = \mathbb{G} (\cdot , z )$ in $\Omega \textbackslash (\overline{D} \cup \{z\})$. By interior elliptic regularity, $v_z$ is continuous at $z \in \Omega \textbackslash \overline{D}$, but $\mathbb{G} ( \cdot ,z)$ has a singularity at $z$. This proves the claim by contradiction, due to the fact that  
$$| v_z (x)| < \infty \quad \text{ whereas}  \quad | \mathbb{G} (x,z) | \rightarrow \infty \quad \text{as} \quad x \rightarrow z.$$

Conversely, we will now assume that $z \in D$. We let $\mu \in H^1 (D)$ be the solution to the following Dirichlet problem in $D$
$$ - \Delta \mu = 0 \enspace \text{in} \enspace D \quad \text{with} \quad \mu \big|^{-}_{\partial D} = \mathbb{G} ( \cdot \, , z) \big|^{+}_{\partial D}$$
Now, define $v_z$ such that 
 \[ v_z = \begin{cases} 
          \mathbb{G}( \cdot \, , z) & \text{in} \enspace \Omega \textbackslash \overline{D}\\
          \mu & \text{in} \enspace D
       \end{cases}
    \]
and we will show $v_z$ satisfy all of the conditions imposed by \thmref{adjoint}. 
By definition we see that $v_z$ is harmonic in $\Omega \textbackslash \partial D$ and that $v_z \in H_0^1(\Omega)$ since there is no jump in the trace across $\partial D$. By construction, we have that $\partial_{\nu} v_z \big|_{\partial \Omega} = \partial_{\nu} \mathbb{G}(\cdot  \, , z )|_{\partial \Omega}$. Now, we need to prove that 
$$g_z =  [\![\partial_\nu v_z ]\!]  \big|_{\partial D} - \gamma v_z \big|_{\partial D}$$
 is in $L^2 ( \partial D)$. To this end, notice that 
 $$ [\![\partial_\nu v_z ]\!] \big|_{\partial D} = \partial_{\nu} \mathbb{G} ( \cdot , z) \big|^+_{\partial D} - \partial_{\nu} \mu \big|^{-}_{\partial D}.$$
Since  $z \in D$ we have that $\mathbb{G} ( \cdot , z ) \in H^{2} ( \Omega \setminus  \overline{D})$. Therefore, we have that 
$$\mathbb{G} ( \cdot , z)\big|^+_{\partial D} \in  H^{3/2} ( \partial D) \quad \text{which implies that } \quad \mu \in H^2(D)$$
by appealing to elliptic regularity (see for e.g. \cite{evans}). By the Neumann Trace Theorem we  obtain that 
$$ [\![\partial_\nu v_z ]\!]  \big|_{\partial D} \in H^{1/2} ( \partial D) \subset  L^2 ( \partial D).$$
Also, it is clear that $ \gamma v_z \big|_{\partial D}\in L^2 ( \partial D)$. We can conclude that $g_z \in L^2 ( \partial D)$ and by appealing to Theorem \ref{adjoint} we have $S^* g_z = \partial_{\nu} \mathbb{G}(\cdot  \, , z )|_{\partial \Omega}$, proving the claim. 
\end{proof}

Using \thmref{greenchar}, we have the following regularized variant of the factorization method for recovering an unknown region $D$ from the knowledge of the difference of the DtN mappings $(\Lambda - \Lambda_0)$.

\begin{theorem}\label{factmain}
The difference of the DtN mappings $(\Lambda - \Lambda_0) : H^{1/2} ( \partial \Omega) \rightarrow H^{-1/2} (\partial \Omega)$ uniquely determines $D$ such that for any $z \in \Omega$ 
$$z \in D \quad \text{if and only if} \quad \liminf_{\alpha \rightarrow 0 } \langle f_{\alpha}^{z} , (\Lambda - \Lambda_{0}) f_{\alpha}^{z} \rangle_{\partial \Omega}$$ 
where $f_{\alpha}^{z}$ is the regularized solution to $(\Lambda - \Lambda_0) f^{z} = \partial_{\nu} \mathbb{G}(\cdot , z ) \big \rvert_{\partial \Omega}$. 
\end{theorem}

This concludes the shape reconstruction problem for an extended region (possibly multiple) via another qualitative method.

\subsection{\textbf{Numerical Validation for the Regularized Factorization Method}}
In this section, we present numerical examples for the regularized factorization method developed in Section \ref{RegFM} for solving the inverse shape problem. Just as in the previous section, our numerical experiments are done in  \texttt{MATLAB} 2020a. For simplicity, we will consider the problem in $\mathbb{R}^2$ where $\Omega$ is the unit disk. Therefore, we again have that the normal derivative of Green's function for the unit disk is given by 
$$ \partial_{\nu(z)} \mathbb{G}\big ( \cdot \, , z \big ) \big|_{\partial \Omega} = - \frac{1}{2 \pi} \bigg[ \frac{1 - |z|^2 }{|z|^2 + 1 - 2 |z| \text{cos}(\cdot \,- \theta_{z})}   \bigg ]$$ 
where $\theta_{z}$ is the polar angle of the sampling point $z \in \Omega$ in polar coordinates.\\

Now, let $\textbf{A} \in \C^{N \times N}$ for $N \in \N$ represent the discretized operator $(\Lambda - \Lambda_0)$ and the vector \textbf{b}$_z = \big [ \partial_{r} \mathbb{G}\big (  \theta_j ,z \big ) \big ]_{j=1}^{N}$. We add random noise to the discretized operator \textbf{A} such that 
$$ \text{\textbf{A}}^{\delta} = \big [ \text{\textbf{A}}_{i,j} \big( 1 + \delta \text{\textbf{E}}_{i,j} \big) \big ]_{i,j=1}^{N} \quad \text{where} \quad \norm{\text{\textbf{E}}}_{2} = 1.$$ 
Furthermore, the matrix \textbf{E} is taken to have random entries. Here $\delta$ is the relative noise level added to the data in the sense that $\|{\text{\textbf{A}}^{\delta} - \text{\textbf{A}}}\|_{2} \leq \delta \|{\text{\textbf{A}}}\|_{2}$. To compute the indicator associated with \thmref{factmain}, we follow \cite{harris1} where it is shown that 
$$ \big( \textbf{f}_z , \textbf{A}^{\delta} \textbf{f}_z \big) =  \sum\limits_{j=1}^{N} \frac{\phi^2(\sigma_j  ; \alpha)}{\sigma_j} \big|({\bf u}_j , {\bf b}_z)\big|^2 .$$
Here $\sigma_j$ and ${\bf u}_j$ denotes the singular values and left singular vectors of the matrix ${\textbf{A}}^{\delta}$, respectively. Also, $\phi(t ; \alpha)$ denotes the filter function defined by the regularization scheme used to solve $ \textbf{A}^{\delta} \textbf{f}_z = \textbf{b}_z$. The filter functions we will use in our examples are given by 
\begin{align}\label{filters}
 \phi(t ; \alpha) =  \frac{t^2}{t^2+\alpha}, \,\, \,  \phi(t ; \alpha) = 1-\left( 1 - \beta t^2\right)^{1/\alpha} \,\, \textrm{ and } \,\,  \displaystyle{  \phi(t ; \alpha)= \left\{\begin{array}{lr} 1, &  t^2\geq \alpha,  \\
 				&  \\
 0,&  t^2 < \alpha
 \end{array} \right.}  
\end{align}
which corresponds to Tikhonov regularization, Landweber iteration (with  $\alpha=1/m$ for some $m \in \N$ and constant $\beta < 1/ \sigma^2_1$) and the Spectral cutoff, respectively.
Using the above expressions, we can recover the unknown region by constructing 
 $$W_{\text{reg}} (z) = \big( \textbf{f}_z , \textbf{A}^{\delta} \textbf{f}_z \big)^{-1} \quad \text{where we plot} \quad W(z) = \frac{W_{\text{reg}} (z)}{\|{W_{\text{reg}} (z)}\|_{\infty}}.$$
 \thmref{factmain} implies that $W(z) \approx 1$ provided that $z \in D$ as well as $W(z) \approx 0$  provided that $ z \notin D$. In the following examples  we use the function $W(z)$ to visualize the defective region. We will provided examples for the different regularization filters given in \eqref{filters}.

\noindent\textbf{Numerical reconstruction of a circular region:}\\
We assume $\partial D$ is given by $\rho (\text{cos} ( \theta ), \text{sin} (\theta))$ for some constant $\rho \in (0,1)$. Since $\Omega$ is assumed to be the unit disk in $\mathbb{R}^2$, we make the ansatz that the electrostatic potential $u(r,\theta)$ has the following series representation
\begin{equation}\label{taurus}
u(r , \theta) = a_0 + b_0 \, \text{ln} \, r + \sum_{|n|=1}^{\infty} \left[ a_n r^{|n|} + b_n r^{-|n|} \right] \text{e}^{\text{i}n \theta} \quad \text{in} \quad \Omega \backslash D
\end{equation}
which is harmonic in the annular region and also
$$u(r , \theta) = c_0 + \sum_{|n|=1}^{\infty} c_n r^{|n|} \text{e}^{\text{i}n \theta} \quad \text{in} \quad D$$
which is harmonic in the circular region.

For simplicity, we assume that the transmission parameter $\gamma >0$ is constant. Thus, we are able to determine the Fourier coefficients $a_n$ and $b_n$ by using the boundary conditions at $r=1$ and $r = \rho$ given by
$$ u(1, \theta ) = f(\theta), \quad u^{+} (\rho , \theta ) = u^{-} ( \rho , \theta ), \quad \text{and} \quad \partial_{r} u^{+} (\rho , \theta ) - \partial_{r} u^{-} (\rho , \theta ) = \gamma u(\rho , \theta).$$
We let $f_n$ for $n \in \mathbb{Z}$ denote the Fourier coefficients for the voltage $f$. Note, that the boundary condition at $r=1$ above gives that $$ a_0 = f_0 \quad \text{and} \quad a_n + b_n = f_n \quad \text{for all} \enspace n \neq 0.$$
The first boundary conditions at $r = \rho$ give that  
$$ b_0 = \frac{\gamma \rho}{1 - \gamma \rho \text{ln} \, \rho} f_0 \quad \text{and} \quad b_n = \rho^{2|n|} (c_n  - a_n ). $$ Using the Robin transmission condition, and after some calculations we get that 
$$a_n = \bigg ( \frac{2|n| + \gamma \rho}{2|n| + \gamma \rho (1 - \rho^{2|n|})}    \bigg ) f_n \quad \text{and} \quad  b_n = \left( \frac{- \gamma \rho^{2|n|+1}}{2|n| + \gamma \rho (1 - \rho^{2|n|})} \right) f_n \quad \text{for all} \enspace  n \neq 0.$$ 
Plugging the sequences into \eqref{taurus} gives that the corresponding current on the boundary of the unit disk is given by
\begin{equation}\label{unormal}
\partial_{r} u(1 , \theta ) = \sigma_0 f_0 + \sum_{|n|=1}^{\infty} |n| \sigma_n f_n \text{e}^{\text{i}n \theta}
\end{equation}
where 
$$\sigma_0 = \frac{\gamma \rho}{1 - \gamma \text{ln} \, \rho} \quad \text{and} \quad \sigma_n = \dfrac{2|n| + \gamma \rho (1 + \rho (1+\rho^{2|n|}))}{2|n| + \gamma \rho (1 - \rho (1+\rho^{2|n|}))} \quad \text{for all} \quad n \neq 0.$$
It is clear that the electrostatic potential and subsequent current for the material without a defective region is given by 
\begin{equation}\label{u0normal}
u_0 (r , \theta ) = f_0 + \sum_{|n|=1}^{\infty} f_n r^{|n|} \text{e}^{\text{i}n \theta} \quad \text{and} \quad \partial_{r} u_0 (1 , \theta ) = \sum_{|n|=1}^{\infty} |n| f_n \text{e}^{\text{i}n \theta}.
\end{equation}
Subtracting equation \eqref{u0normal} from \eqref{unormal} gives a series representation of the current gap operator. By interchanging summation with integration we obtain 
$$ (\Lambda - \Lambda_0)f = \frac{1}{2 \pi} \int_{0}^{2 \pi} K (\theta , \phi) f ( \phi) \, \text{d} \phi \quad \text{where} \quad K(\theta , \phi ) = \sigma_0 + \sum_{|n|=1}^{\infty} |n| (\sigma_n - 1) \text{e}^{\text{i}n (\theta - \phi)}.$$

We now introduce a theorem regarding the convergence of the truncated series approximation for the above integral operator.

\begin{theorem}\label{convergence}
Let $(\Lambda - \Lambda_0)_N : H^{1/2} ( 0 , 2 \pi) \rightarrow H^{-1/2} (0 , 2 \pi)$ be the truncated series approximation of $(\Lambda - \Lambda_0)$. Then we have that $ \|{(\Lambda - \Lambda_0) - (\Lambda - \Lambda_0)_N}\| \leq \frac{C \rho^{2 (N+1)}}{\sqrt{N+1}}.$
\end{theorem}
\begin{proof}
To prove the claim, consider $[(\Lambda - \Lambda_0) - (\Lambda - \Lambda_0)_N]f = \displaystyle \sum_{|n|=N+1}^{\infty} |n| f_n (\sigma_n -1) \text{e}^{\text{i}n \theta} .$
We now use the Cauchy-Schwarz inequality in $\ell^{2}$
\begin{align*}
	\left| \big[(\Lambda - \Lambda_0) - (\Lambda - \Lambda_0)_N \big]f \right|^2  &\leq \bigg( \sum_{|n|=N+1}^{\infty} (\sigma_n - 1)^2 |n| |\text{e}^{\text{i}n \theta}|^{2} \bigg ) \bigg( \sum_{|n|=N+1}^{\infty} |n| |f_n|^2 \bigg) \\
	& \leq \norm{f}_{H^{1/2}(0 , 2 \pi)}^{2} \, \bigg( \sum_{|n|=N+1}^{\infty} (\sigma_n - 1)^2 |n| \bigg ).
\end{align*}
After some calculations we have that $(\sigma_n - 1)^2 |n| \leq \dfrac{\gamma^2 \rho^{2(2|n|+1)}}{|n|},$
which gives that 
$$\norm{[(\Lambda - \Lambda_0) - (\Lambda - \Lambda_0)_N]f}_{\infty} \leq C_{\gamma , \rho} \norm{f}_{H^{1/2}(0 , 2 \pi)} \frac{\rho^{2(N+1)}}{\sqrt{N+1}} .$$
From this, we obtain our result by using the fact that the $H^{-1/2} (0, 2 \pi)$--norm is bounded by the $L^{\infty} ( 0 , 2 \pi )$--norm.
\end{proof}
\thmref{convergence} demonstrates that the convergence for the approximation is slightly better than geometric. Thus, we do not need many terms to approximate the kernel function and obtain desirable results.\\

\noindent{\bf Example 1: recovering a circular region}\\
We approximate the kernel function $K(\theta , \phi )$ given above  by truncating the series for $|n|=1, \hdots , 10$. With this, we then discretize the truncated integral operator by a 64 equally spaced grid on $[0 ,2 \pi)$ using a collocation method. 

In Figure \ref{recon-circle}, we take $\rho = 0.5$ and $\delta = 0.05$ which corresponds to $5 \% $ relative random noise added to the data. Here the regularization scheme is taken to be the Spectral cut-off where the regularization parameter $\alpha=10^{-7}$. The dotted lines are the boundaries of $\partial \Omega$ and $\partial D$ with the solid line being the approximation via the level curve.
\begin{figure}[ht]
\centering 
\includegraphics[scale=0.39]{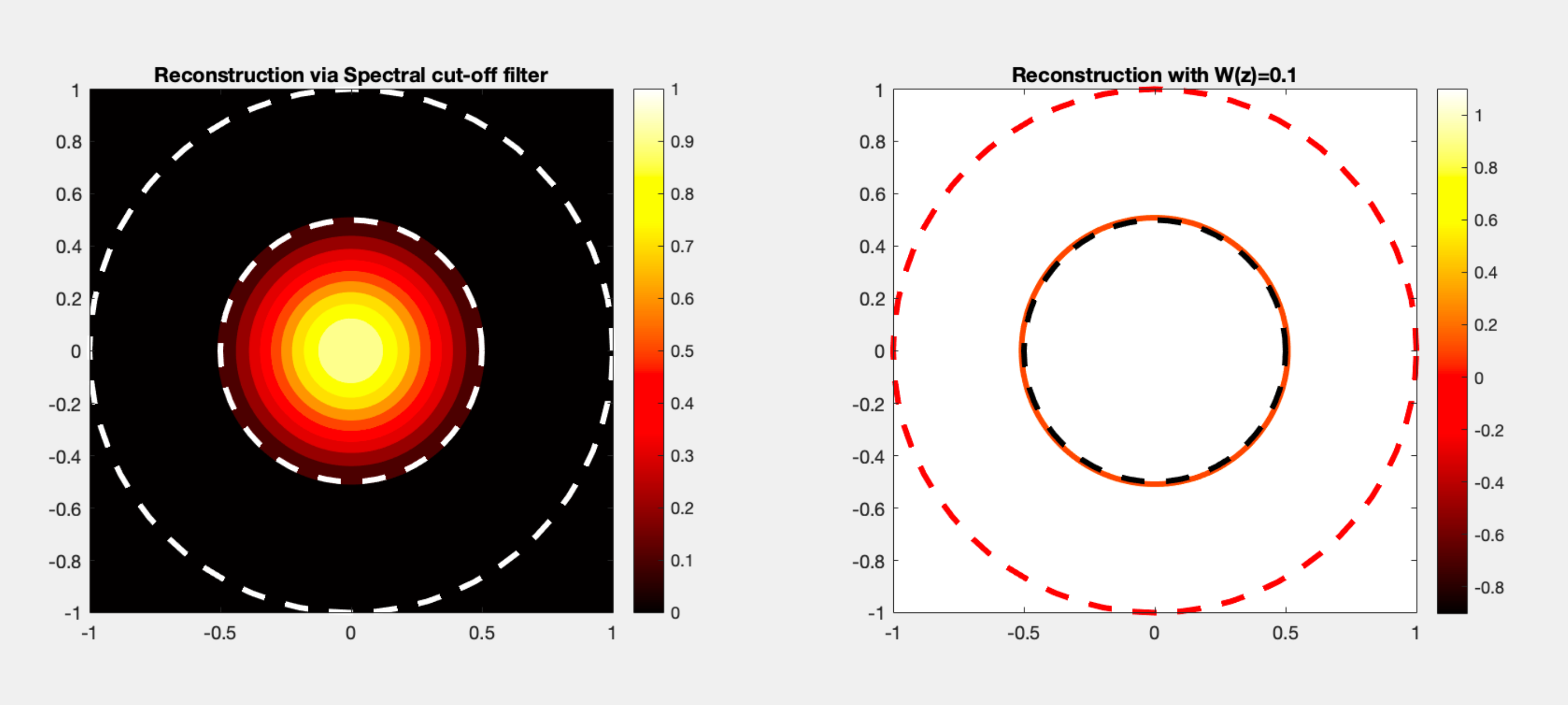}
\caption{Reconstruction of a circular region with $\rho=0.5$ via the regularized factorization method. Contour plot of $W(z)$ on the left and level curve when $W(z)= 0.1$ on the right.}
\label{recon-circle}
\end{figure}

In Figure \ref{recon-circle2}, we take $\rho = 0.25$ and $\delta = 0.02$ which corresponds to $2 \% $ relative random noise added to the data. Here the regularization scheme is taken to be the Tikhonov regularization where the regularization parameter $\alpha=10^{-7}$. The dotted lines are the boundaries of $\partial \Omega$ and $\partial D$ with the solid line being the approximation via the level curve.\\
\begin{figure}[ht]
\centering 
\includegraphics[scale=0.39]{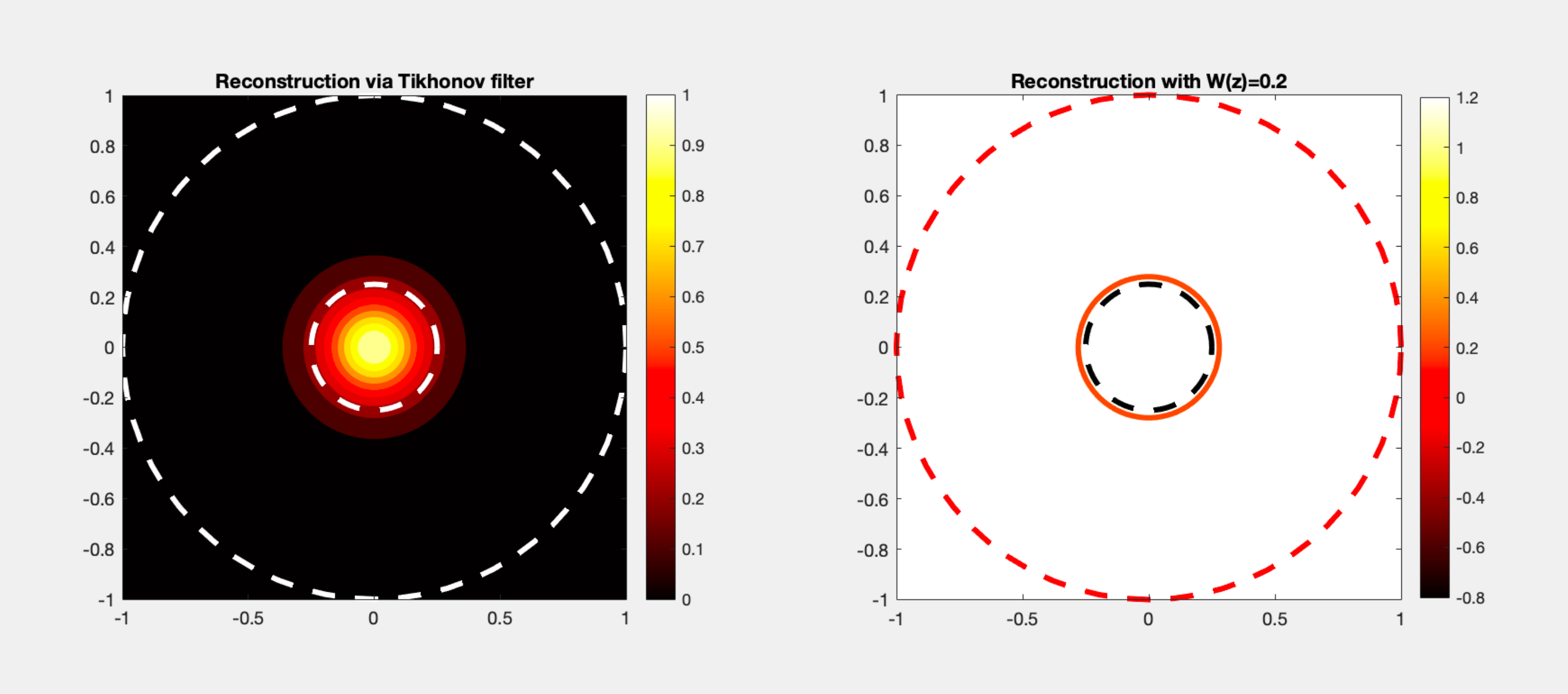}
\caption{Reconstruction of a circular region with $\rho=0.25$ via the regularized factorization method. Contour plot of $W(z)$ on the left and level curve when $W(z)= 0.2$ on the right.}
\label{recon-circle2}
\end{figure}

\noindent\textbf{Numerical reconstruction of a general region:}\\
We will now provide some examples for recovering a region $D$ provided that the boundary of $\partial D$ has the representation in polar coordinates given by
$$\partial D =\big\{ \rho(\theta) (\cos\theta, \sin\theta ) \,  \textrm{ where } \, 0\leq \theta < 2\pi \big\}.$$ 
In our examples, we take $0<\rho(\theta) < 1$ to be a $2\pi$--periodic smooth function. To apply our main result, we need to compute the current gap operator $(\Lambda - \Lambda_0)$. To this end, we compute the mapping $f \longmapsto (\Lambda - \Lambda_0 )f$ where the data $f = \text{e}^{\text{i}n \theta}$ where $|n| = 0, \hdots , 30$. We pick these functions since they form a basis for $H^{1/2} ( \partial \Omega) = H_{\text{per}}^{1/2} [ 0 , 2 \pi]$. Recall, that for any $f \in H^{1/2} (\partial \Omega)$ we have that $u-u_0 \in H^1_0(\Omega)$ satisfies
 \begin{equation}\label{u-u0}
- \Delta (u - u_0) = 0 \quad \text{in} \quad \Omega \backslash \partial D \quad \text{with} \quad  [\![\partial_\nu (u-u_0) ]\!] \big|_{\partial D} = \gamma( u - u_0 )\big|_{\partial D}  + \gamma u_0 \big|_{\partial D} 
\end{equation}
where $u_0 (x , f) = |x|^n \text{e}^{\text{i}n \theta}$. For all the preceding examples we will take the transmission parameter to be given by 
$${\displaystyle \gamma \big( x(\theta) \big) = \frac{1}{4+\text{exp} \big(\cos(\theta) \big)}} .$$
In order to solve \eqref{u-u0} for $u-u_0$ we use the variational formulation with the spectral method presented in \cite{harris3}. Once we have a numerical approximation of $u - u_0$ given by the basis function of the spectral method we have that $(\Lambda - \Lambda_0)f=\partial_r (u - u_0)(1,\theta)$. \\

\noindent{\bf Example 2: recovering an acorn shaped region}\\
In Figure \ref{recon-acorn}, we take $\rho (\theta) = 0.25 \big(1+0.15\cos(3 \theta)\big)$ and $\delta = 0.02$ which corresponds to $2 \% $ relative random noise added to the data. Here the regularization scheme is taken to be the Tikhonov regularization where the regularization parameter $\alpha=10^{-5}$. The dotted lines are the boundaries of $\partial \Omega$ and $\partial D$ with the solid line is the approximation via the level curve.
\begin{figure}[ht]
\centering 
\includegraphics[scale=0.33]{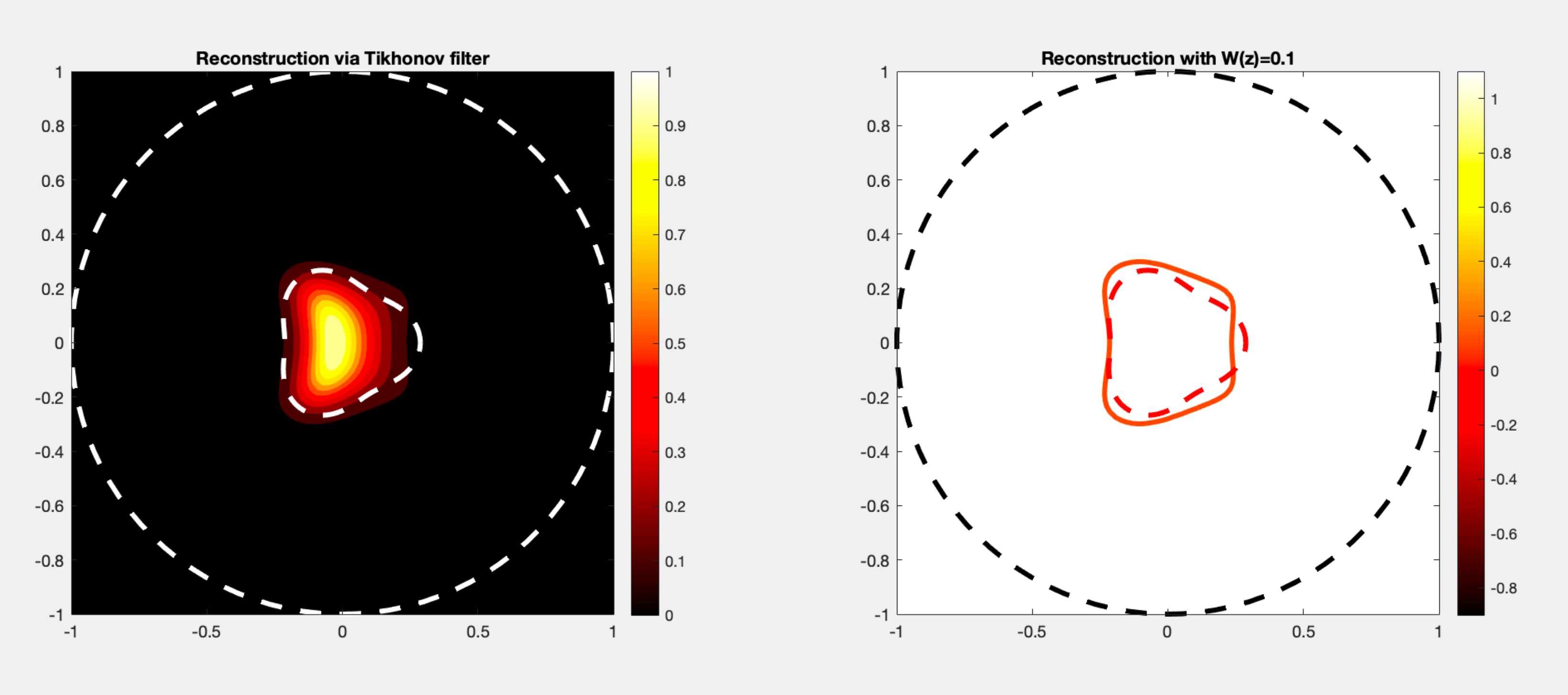}
\caption{Reconstruction of an acorn shaped region via the regularized factorization method. Contour plot of $W(z)$ on the left and level curve when $W(z)= 0.1$ on the right.}
\label{recon-acorn}
\end{figure}

In Figure \ref{compair-acorn}, we again take $\rho (\theta) = 0.25 \big(1+0.15\cos(3\theta)\big)$ and $\delta = 0.02$ which corresponds to $2 \% $ relative random noise added to the data. Here, we  compare the reconstructions using the Tikhonov filter function and Landweber filter function given in \eqref{filters} with $\alpha=10^{-5}$. The dotted lines are the boundaries of $\partial \Omega$ and $\partial D$.\\
\begin{figure}[ht]
\centering 
\includegraphics[scale=0.33]{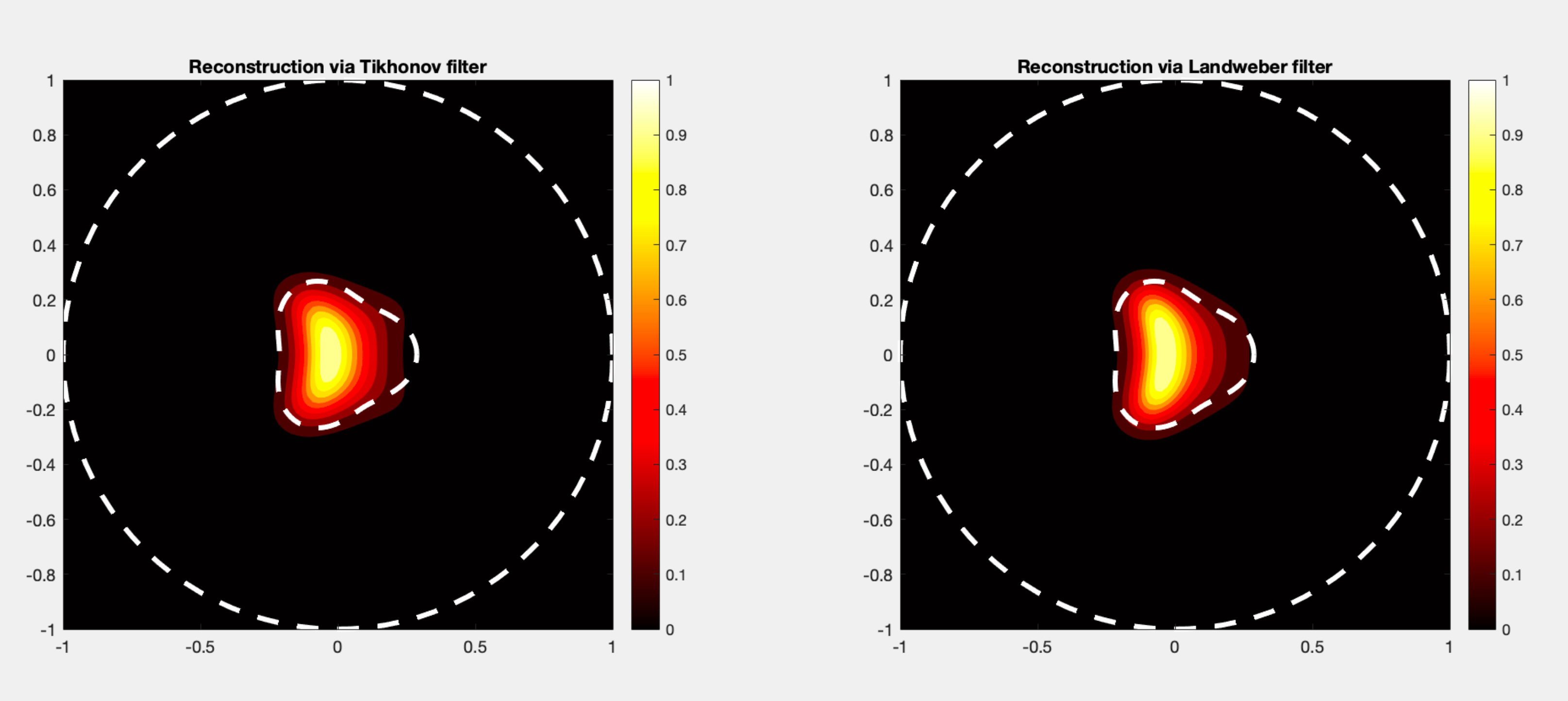}
\caption{Reconstruction of an acorn shaped region via the regularized factorization method. Contour plot of $W(z)$ with the Tikhonov filter on the left and the contour plot of $W(z)$ with the Landweber filter on the right.}
\label{compair-acorn}
\end{figure}

\noindent{\bf Example 3: recovering a star shaped region}\\
In Figure \ref{recon-star}, we take $\rho (\theta) = 0.25 \big(2+0.3 \cos(5 \theta)\big)$ and $\delta = 0.08$ which corresponds to $8 \% $ relative random noise added to the data. Here the regularization scheme is taken to be the Tikhonov regularization where the regularization parameter $\alpha=10^{-5}$. The dotted lines are the boundaries of $\partial \Omega$ and $\partial D$ with the solid line is the approximation via the level curve.
\begin{figure}[ht]
\centering 
\includegraphics[scale=0.33]{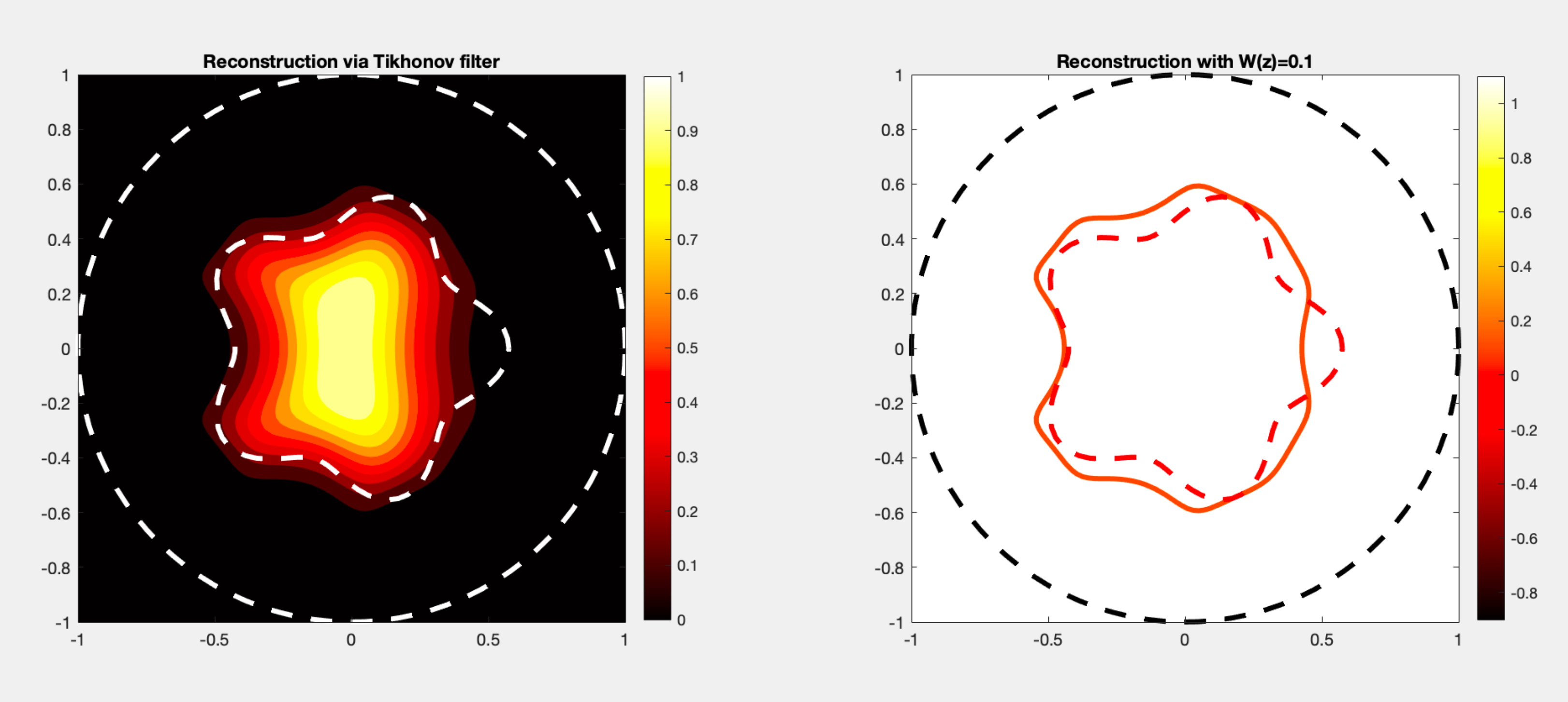}
\caption{Reconstruction of a star shaped region via the regularized factorization method. Contour plot of $W(z)$ on the left and level curve when $W(z)= 0.1$ on the right.}
\label{recon-star}
\end{figure}

In Figure \ref{compair-star}, we again take $\rho (\theta)  =0.25 \big(2+0.3 \cos(5 \theta)\big)$ and $\delta = 0.08$ which corresponds to $8 \% $ relative random noise added to the data. Here, we compare the reconstructions using the Tikhonov filter function and Landweber filter function given in \eqref{filters} with $\alpha=10^{-5}$. The dotted lines are the boundaries of $\partial \Omega$ and $\partial D$.\\
\begin{figure}[ht]
\centering 
\includegraphics[scale=0.33]{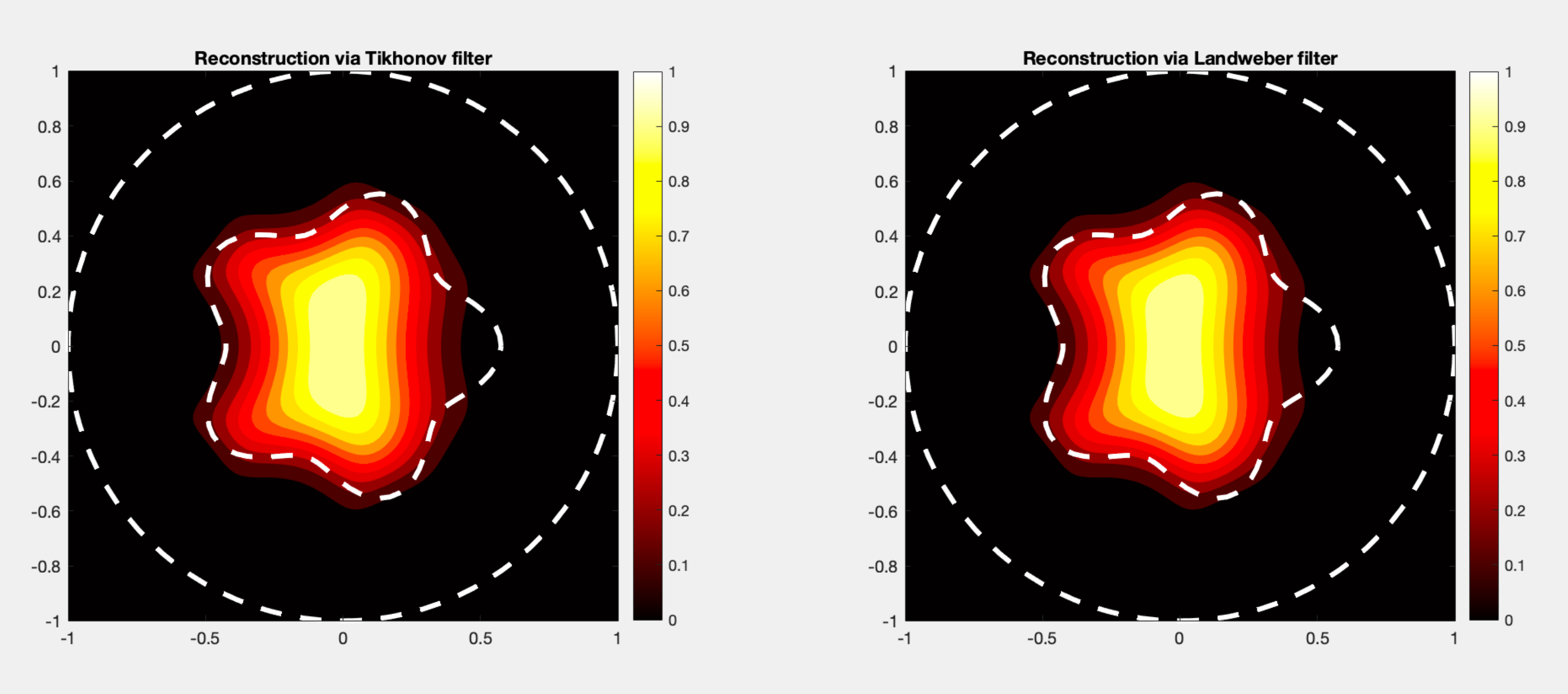}
\caption{Reconstruction of a star shaped region via the regularized factorization method. Contour plot of $W(z)$ with the Tikhonov filter on the left and the contour plot of $W(z)$ with the Landweber filter on the right.}
\label{compair-star}
\end{figure}

Notice that by the examples provided here, it does not seem that the reconstruction is sensitive to the regularization scheme. In the Figures \ref{recon-circle}--\ref{compair-star}, we see that there is little to no difference in the reconstruction when different filter functions are used. For the examples, where the Landweber filter was used we took $\beta =1/ 2\sigma^2_1$ in the reconstruction. Also, we have picked the regularization parameter ad hoc in our examples. In practice, a discrepancy principle would we be used to pick an optimal regularization parameter. 

\section{Conclusions}\label{end}
In this paper, we have studied two qualitative methods for the inverse shape problem in EIT with a Robin transmission condition.  We have analyzed the MUSIC algorithm of small volume regions and the regularized factorization method for extended regions. In both cases, we have derived imaging functionals to recover the region of interest $D$ using current gap operator. This allows for fast and accurate reconstruction with little to no a priori knowledge of $D$. A future direction for this project can be to study the inverse parameter problem and derive a non-iterative method for recovering $\gamma$. One could also consider, studying the direct sampling method (see for e.g. \cite{DSM-DOT,DSM-EIT,DSM-nf}) for this problem. Also, the analysis of this inverse problem for a generalized Robin condition is still open.  \\

\noindent{\bf Acknowledgments:} The research of I. Harris is partially supported by the NSF DMS Grant 2107891. The Authors would also like to thank R. Ba\~nuelos and K. Datchev for useful discussions on the topic.


\end{document}